\documentclass[11pt]{article} \usepackage[utf8]{inputenc}
\usepackage[T1]{fontenc}
\usepackage{lmodern}
\usepackage[english]{babel}
\usepackage{showlabels}

\usepackage[a4paper,vmargin={3.5cm,3.5cm},hmargin={2.5cm,2.5cm}]{geometry}
\usepackage[font={small,sf}, labelfont={sf,bf}, margin=1cm]{caption}

\usepackage[pdftex,colorlinks=true]{hyperref}
\usepackage[expansion]{microtype}
\usepackage[pdftex]{color,graphicx}
\usepackage{amsmath,amsfonts,amssymb,amsthm,mathrsfs,ulem}
\usepackage{stackrel,framed}
\usepackage{subfig}
\usepackage{stmaryrd} 

\theoremstyle{plain}

\newtheorem{coro}{Corollary}
\newtheorem{theo}[coro]{Theorem}
\newtheorem{prop}[coro]{Proposition}

\newtheorem{lemm}[coro]{Lemma}

\graphicspath{{images/}}

\renewcommand{\leq}{\leqslant}
\renewcommand{\geq}{\geqslant}

\hypersetup{
    pdfauthor   = {Laurent Ménard},%
    pdftitle    = {Site percolation on the UIPT}}

\begin{document}
\normalem
\title{\bf Percolation probability and critical exponents\\ for site percolation on the UIPT}
\author{{Laurent Ménard}\footnote{{New York University Shanghai, 200122 Shanghai, China,} and {Modal'X, UMR CNRS 9023, UPL, Univ. Paris Nanterre, F92000 Nanterre, France}. \href{mailto:laurent.menard@normalesup.org}{laurent.menard@normalesup.org}}}
\date{January 27, 2022}
\maketitle


\begin{abstract}
We derive three critical exponents for Bernoulli site percolation on the on the Uniform Infinite Planar Triangulation (UIPT). First we compute explicitly the probability that the root cluster is infinite. As a consequence, we show that the off-critical exponent for site percolation on the UIPT is $\beta = 1/2$. Then we establish an integral formula for the generating function of the number of vertices in the root cluster. We use this formula to prove that, at criticality, the probability that the root cluster has at least $n$ vertices decays like $n^{-1/7}$. Finally, we also derive an expression for the law of the perimeter of the root cluster and use it to establish that, at criticality, the probability that the perimeter of the root cluster is equal to $n$ decays like $n^{-4/3}$. Among these three exponents, only the last one was previously known. Our main tools are the so-called gasket decomposition of percolation clusters, generic properties of random Boltzmann maps, as well as analytic combinatorics.

\bigskip

\noindent{\bf MSC 2010 Classification:}
05A15, 
05A16, 
05C12, 
05C30, 
60C05, 
60D05, 
60K35. 
\end{abstract}

\bigskip


\section{Introduction}

Percolation on random planar maps has been studied intensively since the pioneering work of Angel \cite{AngelPerco}. The main feature of random planar maps making this study so fruitful is the spatial Markov property. It can be used with two different approaches. The first approach is to perform an exploration process of percolation interfaces with the so-called peeling process. This is the approach developed by Angel \cite{AngelPerco} to prove that the threshold for Bernoulli site percolation on the Uniform Infinite Planar Triangulation (UIPT, the limit in law of large uniform random triangulations for the local topology \cite{AngelSchramm}) is $1/2$. This approach has been later used by several authors to study other models of percolation on maps, see for example \cite{AngelCurien,BC21,GMSS,MN,Richier} and the references therein. The second approach is more global and consists on decomposing the map into the cluster of the root vertex and pieces filling the faces of this cluster. Such a decomposition is often called the Gasket decomposition, see for instance the works of Borot, Bouttier, Duplantier and Guitter \cite{BBD,BBGa,BBGc}. This second approach has been used very recently to study percolation on random finite triangulations by Bernardi, Curien and Miermont \cite{BCM}, following the previous work by Curien and Kortchemski \cite{CuKo}; and on other natural models of random finite planar maps by Curien and Richier\cite{CR}.

\medskip

This work builds on the article by Bernardi-Curien-Miermont \cite{BCM} to study site percolation on the UIPT. Our first main result is an explicit formula for the probability that percolation from the root occurs.

\begin{theo} \label{th:offcrit}
Let $\mathbb P_\infty^p$ denote the law of the UIPT, with vertices colored black with probability $p$ and white with probability $1-p$, and conditioned on the event where the root edge has both end vertices colored black. Let $\mathfrak C$ denote the site percolation cluster of the root vertex under $\mathbb P_\infty^p$. Then, for every $p \in [0,1]$, we have
\begin{align*}
\mathbb P_\infty^{p} (|\mathfrak C| = \infty) 
& = 
2 \,
\frac{\sqrt{2p -1} \,
\left(\sqrt 3 - 
\cos^3 \left(\frac{2}{3}\mathrm{arccos}(\sqrt p) \right) \right)
\,  
\left(  \cos \left(\frac{2}{3}\mathrm{arccos}(\sqrt p) \right) \right)^{3/2}
+ 
p(2p-1)}{p \, (2 \sqrt 3 - 3 (1-p))} \,
\mathbf 1_{p \geq 1/2}.
\end{align*}
In particular, the critical exponent is $\beta = 1/2$: as $p \to 1/2^+$ one has
\begin{align*}
\mathbb P_\infty^{p} (|\mathfrak C| = \infty) 
& = 3^{1/4} \, \frac{15}{26} \, \left(1 +\frac{4\sqrt{3}}{3} \right) \sqrt{p - \frac{1}{2}} + \mathcal O \left( p - \frac{1}{2} \right).
\end{align*}
\end{theo}
\begin{figure}[ht!]
\centering
\includegraphics[width=0.7\linewidth]{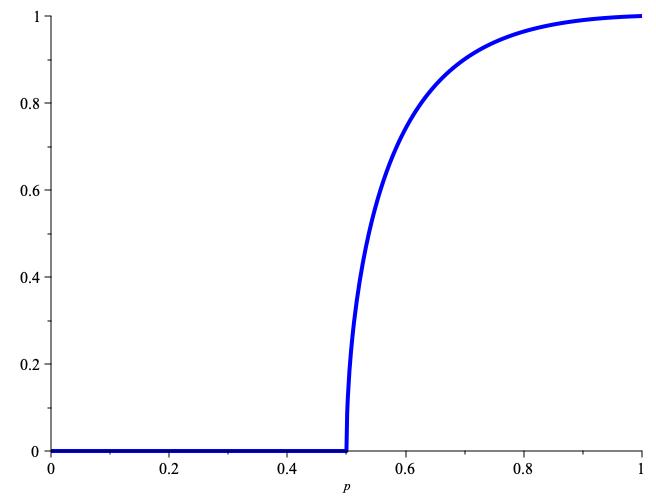}
\caption{\label{fig:plot}Plot of the probability $\mathbb P_\infty^p  (|\mathfrak C| = \infty)$ for $p\in [0,1]$.}
\end{figure}

\bigskip

To the best of our knowledge, this is the first formula of this type and the first calculation of the critical exponent $\beta$ for percolation on the UIPT. A similar formula was obtained for percolation on the Uniform Infinite Half-Plane Triangulation (UIHPT) by Angel and Curien \cite{AngelCurien} using the peeling process. They also calculate explicitly the probability that percolation from the root occurs and then obtain $\beta = 1$ in the UIHPT setting. The relation between these two exponents remains quite mysterious, and we do not know of any strategy to obtain the exponent $\beta$ without first computing explicitly the probability of percolation.

\bigskip

Our second main result gathers estimates for the volume and perimeter of critical percolation clusters.
\begin{theo} \label{th:volume}
With the notation of Theorem \ref{th:offcrit} we have for an explicit constant $\kappa >0$:
\begin{align*}
\mathbb P_\infty^{1/2} \left( |V(\mathfrak C)| \geq  n \right)  \underset{n \to \infty}{\sim} \kappa \, n^{-1/7}.
\end{align*}
Furthermore, if $\partial \mathfrak C$ denotes the root face of $\mathfrak C$, then there is an explicit constant $\kappa' >0$ such that
\begin{align*}
\mathbb P_\infty^{1/2} \left( |V(\partial \mathfrak C)| = n \right)  \underset{n \to \infty}{\sim} \kappa' \, n^{-4/3}.
\end{align*}
\end{theo}

The perimeter exponent $4/3$ was established by Curien and Kortchemski \cite{CuKo} using the gasket decomposition but with a different approach than the present work.
The exponent $1/7$ was conjectured in \cite{GMSS} using heuristics for the peeling process, and the present work is the first time it is established rigorously. Previous works that computed volume exponents for critical percolation models on infinite random planar maps did so for cluster hulls (part of the maps separated from infinity by a percolation interface). The reason being that all these works use variations around the peeling process, which is particularly well suited to study percolation interfaces -- and therefore hulls -- but not so useful to study the geometry of the clusters themselves. For instance, let $\mathfrak H$ denote the hull of the root cluster $\mathfrak C$ in the UIPT; i.e. $\mathfrak H$ is the complement of the only infinite connected component of $\mathbf T_\infty \setminus \mathfrak C$, where $\mathbf T_\infty$ is the UIPT with a critical site percolation configuration. Gorny, Maurel-Segala and Singh \cite{GMSS} proved that
\[
\mathbb P \left(  |V(\mathfrak H)| \geq  n \right)  \asymp n^{-1/8} \quad \text{and} \quad \mathbb P \left(  |V(\partial \mathfrak H)| \geq  n \right)  \asymp n^{-1/6},
\]
where $u_n \asymp v_n$ means that $u_n/v_n$ is bounded.

\paragraph{Main ingredients and organization of the paper.} For every site percolated finite rooted planar triangulation $\mathfrak t$, we denote by $\mathrm v_\circ (\mathfrak t)$ and  $\mathrm v_\bullet (\mathfrak t)$ its number of white (resp. black) vertices, and by $\mathrm e(\mathfrak t)$ its number of edges. For $p \in (0,1)$ and for $t>0$ small enough, we can consider a random finite percolated triangulation $\mathbf{t}$ whose law gives a probability proportional to $t^{\mathrm e(\mathfrak t)} \, p^{\mathrm v_\bullet(\mathfrak t)} (1-p)^{\mathrm v_\circ (\mathfrak t)}$ to every finite triangulation $\mathfrak t$. Let us denote by $\mathcal Z$ the partition function of this model and by $\mathbb P^p$ the corresponding probability:
\[
\mathcal Z (p,t) := \sum_{\mathfrak t} t^{\mathrm e(\mathfrak t)} \, p^{\mathrm v_\bullet(\mathfrak t)} (1-p)^{\mathrm v_\circ (\mathfrak t)} < +\infty, \quad \text{and} \quad \mathbb P^p \left( \mathbf t= \mathfrak t\right) = \frac{t^{\mathrm e(\mathfrak t)} \, p^{\mathrm v_\bullet(\mathfrak t)} (1-p)^{\mathrm v_\circ (\mathfrak t)}}{\mathcal Z (p,t)}.
\]
The partition function $\mathcal Z$ and its generalizations to triangulations with a boundary (with additional parameters counting respectively the number of boundary vertices and boundary edges) are studied in Section \ref{sec:genser}. In particular, we establish a rational parametrization for the generalized partition function. This parametrization significantly simplifies the study of its analytic properties started in \cite{BCM}.

\bigskip

The gasket decomposition resides in the following statement. There exists an explicit sequence of positive numbers $(q_k(p,t))_{k \geq 1}$ such that, for every finite non atomic map $\mathfrak m$, the probability that the root cluster $\mathfrak C (\mathbf t)$ of the random triangulation $\mathbf t$ is equal to $\mathfrak m$ is given by
\[
\mathbb P^p \left( \mathfrak C (\mathbf t) = \mathfrak m \right) = \frac{\prod_{f\in \mathrm{Faces}(\mathfrak m)} q_{\mathrm{deg}(f)} (p,t)}{\mathcal Z (p,t)}.
\]
In this sense, the root cluster is a Boltzmann random planar map associated to the weight sequence $(q_k(p,t))_{k \geq 1}$. The properties of such random maps depend on the asymptotic behavior in $k$ of the weight sequence, which can be inferred from the generating function of the weights. In our case, the weight generating function is closely related to the generalized partition function of percolated triangulations counted by edges, boundary edges, and boundary vertices. A crucial consequence of this relation is that it allows to calculate explicitly the so-called pointed disk generating function of the root cluster in terms of the singularities of the generalized partition function of percolated triangulations in Equations \eqref{eq:expWbullet} and \eqref{eq:cpm}. See Section \ref{sec:BPM} for details.

\bigskip

To study the origin cluster of the UIPT, we can condition the random triangulation $\mathbf t$ to have $3n$ edges. By continuity for the local topology, this gives
\begin{equation} \label{eq:limlaw}
\mathbb P^p \left( \mathfrak C (\mathbf t) = \mathfrak m \middle| \mathrm e(\mathbf t) = 3n \right) = \frac{[t^{3n}]\prod_{f\in \mathrm{Faces}(\mathfrak m)} q_{\mathrm{deg}(f)} (p,t)}{[t^{3n}] \mathcal Z (p,t)} \underset{n \to \infty}{\rightarrow} \mathbb P_\infty^{p} (\mathfrak C = \mathfrak m).
\end{equation}
With a careful study of the dependency in $t$ of the weight sequence $(q_k(p,t))_{k \geq 1}$ performed in Section \ref{sec:teclem}, we are able to compute the above limit, giving the law of the root cluster in the UIPT on the event where it is finite.

We are then able to establish an integral formula for the sum of the limiting probability over every finite map, which is the probability that the root cluster of the UIPT is finite. See Proposition~\ref{prop:integralProbab}. In particular, this calculation uses the explicit universal form of the  pointed disk generating functions and cylinder generating functions of Boltzmann maps. With the help of our rational parametrizations, we can then calculate explicitely our integral formula for the probability that the root cluster is finite and establish Theorem \ref{th:offcrit}. This is performed in Section \ref{sec:main}.

The formula we obtain for the limit \eqref{eq:limlaw} is quite easy to sum over maps with the same perimeter. As a consequence we are also able to calculate explicitely the law of the perimeter of the root cluster and obtain the second statement of Theorem \ref{th:volume} on the perimeter of the root cluster. This is done at the beginning of Section \ref{sec:vol}.

\bigskip

Finally, to compute the tail probability of the number of vertices in the root cluster at criticality ($p=1/2$), we establish an integral formula for the generating function $\mathbb E_\infty^{p} \left[ g^{|V(\mathfrak C )|} \right]$ in Section \ref{sec:vol}. See in particular identity \eqref{eq:Volgint}, which also originates from our explicit formula for the limit \eqref{eq:limlaw}. The expression we get involves two quantities. The first is the generating series derived from the asymptotics of the coefficients in $t$ of the weights $(q_k(p,t))_{k \geq 1}$ studied in Section \ref{sec:teclem} for which we have an explicit parametric expression. The second quantity involves the pointed disk generating functions of Boltzmann maps with modified weight sequence $(g^{(k-2)/2} q_k(p,t))_{k \geq 1}$. We can analyze the behavior as $g \to 1^-$ of these modified pointed disk generating functions using the bivariate generating functions associated to the Bouttier-Di Francesco-Guitter bijection \cite{BDFG} presented in Section \ref{sec:criteq} and their singularities obtained in Section \ref{sec:gfsing}. We put all that together in Section \ref{sec:vol} to study the singular behavior as $g \to 1^-$ of $\mathbb E_\infty^{1/2} \left[ g^{|V(\mathfrak C)|} \right]$ and prove Theorem \ref{th:volume}.

\paragraph{On the robustness of our approach.} We believe that it should be possible to adapt the strategy of this paper to other models of random planar maps and other statistical models. Indeed, what is needed first is the gasket decomposition; which exists for example for percolation on other models of maps \cite{CR}, or for the $\mathcal O (n)$ model on maps \cite{BBD,BBGa}. We then need information on the generalized generating series of maps with a boundary and their singularities. In the present work we have explicit parametric expressions to simplify calculations, but we really only need to identify the nature of the singularities for the critical exponents. Any model for which such information is available should fall into the scope of our method.

In another article \cite{AM}, we derive several critical exponents for the sign clusters in finite and infinite planar triangulations coupled with an Ising model. In particular, we establish in \cite{AM} counterparts for the main results of the present work. Theorem \ref{th:offcrit} and its counterpart for Ising model are unrelated, but Theorem \ref{th:volume} is a particular case of its Ising version at infinite temperature.

\paragraph{Links with other critical exponents for percolation on the UIPT.} The volume exponent $1/7$ of Theorem \ref{th:volume} suggests that the scaling limit of a large critical percolation cluster in the UIPT should be of quantum dimension $7/8$. There is also very strong evidence (see for instance the works of Bernardi-Holden-Sun \cite{BeHoSun} and Holden-Sun \cite{BeHoSun}) that in some sense, this scaling limit is a CLE$_6$ on an independent pure Liouville Quantum Gravity surface. The quantum dimension $7/8$ of the scaling limit of the root cluster and the dimension $91/48$ of the gasket of a CLE$_6$ agree with the KPZ \cite{KPZ} relation:
\[
1 - \frac{1}{2} \frac{91}{48} = \frac{2}{3} \left(1 - \frac{7}{8} \right)^2 + \frac{1}{3} \left(1 - \frac{7}{8} \right).
\]
In a similar fashion, the perimeter exponent $4/3$ of Theorem \ref{th:volume} agrees with the KPZ relation and the dimension $\frac{7}{4}$ of a SLE$_6$ curve \cite{SLEdim}

The value of the critical exponent $\beta$ agrees heuristically with known quantities of the UIPT and Kesten's scaling relations. Indeed, Bernardi-Holden-Sun \cite{BeHoSun} and Holden-Sun \cite{HS} established that the the number of percolation pivotal points in a random triangulation of size $n$ is of order $n^{1/4}$. In this sense, the quantum dimension of the set of pivotal points of  critical percolation on the UIPT should be $1$ (the map itself having dimension $4$). This dimension can also be predicted with the KPZ relation and the dimension $3/4$ of critical percolation pivotal points in Euclidean geometry \cite{GarbanPeteSchramm}. On the other hand, the one arm exponent $\alpha_1$ of critical percolation on the UIPT should be $1/2$ from the quantum dimension $7/8$ of the the clusters. Kesten's scaling relation then states that the dimension of pivotal points should be $\alpha_1/\beta$, giving $1$ with the exponent of Theorem \ref{th:offcrit}.

\paragraph{Acknowledgments.} The author thanks Pierre Nolin for stimulating discussions on critical exponents. This work is partially supported by the ANR grant ProGraM (Projet-ANR-19-CE40-0025) and the Labex MME-DII (ANR11-LBX-0023-01).

\tableofcontents

\section{Generating series} \label{sec:genser}

Let $T$ be the generating series of rooted triangulations with a (not necessarily simple) boundary counted by edges (variable $t$), boundary length (parameter $y$), and boundary vertices (parameter $p$). That is, we define:
\[
T(p,t,y) = \sum_{k\geq 1} \sum_{\mathfrak t \in \mathcal T_k} t^{e(\mathfrak t)} p^{\mathrm v_{\mathrm {out}}(\mathfrak t)} y^k = \sum_{k\geq 1} T_k(p,t) \, y^k,
\]
where $\mathcal T_k$ is the set of all rooted triangulations with a boundary face of degree $k$, and where ${e(\mathfrak t)}$ and  $\mathrm v_{\mathrm {out}}(\mathfrak t)$ denote respectively the total number of edges and number of boundary vertices of the triangulation with a boundary $\mathfrak t$.

From \cite[Lemma 3.1]{BCM}, we have the following equation fo $T(p,t,y)$:
\begin{equation} \label{eq:Ty}
T(p,t,y) = p + y^2 t T^2(p,t,y) + (p-1) t \frac{(T(p,t,y)-p)^2}{p y T(p,t,y)} + \frac{t}{py} \left( T(p,t,y)-p-y \,T_1 (p,t) \right).
\end{equation}
Using the quadratic method, the authors of \cite{BCM} establish the following algebraic equation for $T_1 \equiv T_1(p,t)$ that will be our starting point:
\[
64 \, T_1^3 \, t^5-27 \,p^3 \, t^5-96 \, T_1^2 \, p \, t^4+30\,T_1\, p^2\, t^3+p^3\,t^2+T_1^2\,p\,t-T_1 \, p^2 =0.
\]
Up to a multiplicative constant $p$, the series $T_1$ is simply the generating series of triangulations of the $1$-gon counted by edges and admits a proper rational parametrization:
\begin{lemm} \label{lemm:parU}
Let $U$ be the unique power series in $t^3$ having constant term $0$ and satisfying
\[
t^3 = \hat w(U) := \frac{1}{2} U (1-U) (1-2U).
\]
The series $t\,T_1 (p,t)$  seen as a series in $t^3$ admits the following proper rational parametrization in $U$:
\[
t \, T_1(p,t) = \hat T_1 (p,U) := p \, U \, \frac{1-3U}{1-2U}.
\]
Furthermore, the series $U(t^3)$ has a unique dominant singularity at $t^3 = (t_c)^3 := \frac{\sqrt{3}}{36}$ with the following singular behavior:
\begin{equation} \label{eq:singU}
\begin{split}
U(t^3) = \frac{3-\sqrt{3}}{6} - \frac{\sqrt{2}}{6} \left(1- \left(\frac{t}{t_c}\right)^3 \right)^{1/2} & + \frac{\sqrt{3}}{54} \left(1- \left(\frac{t}{t_c}\right)^3 \right)\\
& - \frac{5\sqrt{2}}{648} \left(1- \left(\frac{t}{t_c}\right)^3 \right)^{3/2} + \mathcal O \left(1- \left(\frac{t}{t_c}\right)^3 \right)^{2}.
\end{split}
\end{equation}
\end{lemm}
\begin{proof}
This result could be qualified as folklore since the series $T_1$ is just the generating series of triangulations of the $1$-gon multiplied by $p$. The rational parametrization given is also the classical one. The fact that $U(t^3)$ is unique comes from the Lagrangean form of the equation that defines it. We can also see that the algebraic equation satisfied by $t \,T_1(t)$ has a unique solution that is power series with constant term $0$. By composition and unicity we can see that indeed $tT_1(t) = \hat T_1(p,U(t^3))$ as power series in $t$ since both verify the algebraic equation and have constant term $0$. The singular behavior of $U$ is also very classical and without difficulties. See the Maple companion file \cite{Maple} for details.
\end{proof}

Injecting the parametrization of $t^3$ and $t \, T_1$ by $U$ in the equation for $T$ then allows to establish a proper rational parametrization for $T$:
\begin{lemm} \label{lemm:parT}
Recall the definition of the power series $U \equiv U(t^3)$ from Lemma \ref{lemm:parU}. Let $V \equiv V(p,U,y)$ be the unique power series in $\mathbb Q[p,U] \llbracket y \rrbracket \subset \mathbb Q[p] \llbracket t^3,y \rrbracket$ with constant term in $y$ equal to $0$ satisfying:
\[
y= \hat y (p,U,V) := \frac{2 \, V \, (2 - 4\,U -V)}{4 \, p \, U \, (1-U) \, (1-2U) + 2 \, U \, (1-3U) \, V + 2 \, (1-3U) V^2 -V^3}.
\]
The series $T (p,t,ty)$ seen as a series in $t^3$ and $y$ is algebraic and admits the following proper rational parametrization in $U$ and $V$
\[
T(p,t,ty) = \hat T(p,U,V) := \frac{4 \, p \, U \, (1-U) \, (1-2U) + 2 \, U \, (1-3U) \, V + 2 \, (1-3U) V^2 -V^3}{4\, U \, (1-U) \, (1-2U)}.
\]

In addition, for any $p \in (0,1]$ and any fixed $t \in (0, t_c]$, the series $T(p,t,ty)$ seen as a series in $y$ has radius of convergence $y_+(p,t)>1$ where it is singular. Furthermore, the series $T(p,t,ty)$ can be analytically continued in the domain $\mathbb C \setminus \left((-\infty ,y_-(p,t)] \cup [y_+(p,t), +\infty) \right)$, where $y_-(p,t)$ is on the negative real line ($-\infty$ included) and is such that $y_-(p,t) < - y_+(p,t)$.
\end{lemm}

\begin{proof}

All computations are available in the Maple companion file \cite{Maple}.

The fact that $V$ is uniquely defined as a power series comes from the Lagrangean form of the equation $V = y \times R (p,U,V)$ with $R$ a rational fraction such that $R (p,U,0) \neq 0$. This form also implies by inductive calculation of the coefficients in $y$ of $V$ that they are all rational in $p$ and $U$.
Similarly, the equation \eqref{eq:Ty} verified by $T(p,t,ty)$ takes the form
\[
\begin{split}
p (T(p,t,ty) - p) & (T(p,t,ty) - (1-p))\\
&= y \, T(p,t,ty) \, \left( tT_1(p,t) - p (T(p,t,ty) - p) - y^2 \, p \, t^3 \, T(p,t,ty)^2 \right).
\end{split}
\]
Here again, by inductive calculation of the coefficients, we can see that this last equation has a unique solution that is a power series in $t^3$ and $y$ with constant term in $y$ equal to $p$. Note that since the series $tT_1(p,t)$ is algebraic, this equation also ensures that $T(p,t,ty)$ is algebraic.
By composition, $\hat T \left(p, U(t^3), V(p,U(t^3),y) \right)$ is a power series in $t^3$ and $y$ with constant term in $y$ equal to $p$. We can verify that it satisfies the same algebraic equation as $T(p,t,ty)$, and therefore the two power series are identical.

\bigskip

Now, fix $t^3 \in (0,t_c]$. The function $V \mapsto \hat y (p,U(t^3),V)$ has poles and stationary points that we can locate. Let us start with the poles. The denominator of $\hat y (p,U(t^3),V)$ is a polynomial of degree three in $V$. It is positive with positive derivative at $V = 0$, and changes signs between $1-2U$ and $2(1-2U)$. Since the coefficient of $V^3$ is $-1$, this leaves two possibilites for the poles of $\hat y (p,U(t^3),V)$: there is always a pole between $1-2U$ and $2(1-2U)$, and either there is no additional real pole or there are two negative poles.

The stationary points of $\hat y (p,U(t^3),V)$ are the roots of the polynomial
$-2V^4 + (-16U + 8)V^3 - 4(3U - 1)(3U - 2)V^2 - 16Up(2U - 1)(U - 1)V - 16Up(U - 1)(2U - 1)^2$, where $U$ stands for $U(t^3)$. By computing the values of the polyomial at $0$, $1-2U$ and $2(1-2U)$, we can see that it has $4$ real roots $V_-(p,U) < 0 < V_+(p,U) \leq 1-2U \leq V_l(p,U) < 2(1-2U) < V_r(p,U)$ (the case of a double root at $1-2U$ only happens when $U=U_c$ and is treated separately in the Maple file). It is also easy to see that $\partial_V \hat y (p,U(t^3),0) >0$, and therefore $y_+(p,t) := \hat y (p, U(t^3), V_+(p,U(t^3))) >0$. We can define $y_-(p,t) := \hat y (p, U(t^3), V_-(p,U(t^3))) < 0$ when $\hat y (p,U(t^3),V)$ has no pole between  $V_-(p,U(t^3))$ and $0$, and $y_-(p,t) := -\infty$ when it has such a pole. By singular inversion, the inverse function $V(p, U(t_3),y)$ of $\hat y (p,U(t^3),V)$ is analytic in the domain $\mathbb C \setminus \left((-\infty ,y_-(p,t)] \cup [y_+(p,t), +\infty) \right)$ and singular at both points $y_\pm(p,t)$ when they are finite.

By composition $y \mapsto T(p,t,ty)$ is analytic in the same domain. Since it has nonnegative coefficients, we know that it is singular at its radius of convergence. Therefore $y_+(p,t)$ is its radius of convergence and $y_-(p,t) \leq - y_+(p,t)$. Checking that $y_-(p,t) < - y_+(p,t)$ just with our current material is cumbersome but could be done. However, we do not need to go through this since this inequality is a direct consequence of equations \eqref{eq:cpm} established with no computations at the end of Section \ref{sec:BPM}. Finally, to see that $y_+(p,t) >1$, we first note that for $p$ fixed, $y_+(p,t)$ is non increasing with respect to $|t|$ and then check $y_+(p,t_c) > 1$ by direct computation.
\end{proof}

\section{The root cluster as a Boltzmann map} \label{sec:BPM}

In this whole section, $p \in (0,1)$ and $t \in (0,t_c]$ are fixed.

\subsection{The weight sequence from \texorpdfstring{\cite{BCM}}{[BCM]} with the edge parameter}

In \cite[Section 2.2]{BCM} it was established that for a random site-percolated triangulation (conditioned on the event where both ends of the root edge are colored in black), the cluster of the root is a random Boltzmann map with weight sequence $\mathbf q (p,t) = \left(q_k(p,t) \right)_{k \geq 1}$ given by
\begin{equation} \label{eq:defqkt}
q_k(p,t) = \frac{1}{p} \left((pt)^{3/2} \delta_{\{k=3\}} + (p t^3)^{k/2} \sum_{l \geq 0} \binom{k+l-1}{k-1}  [y^l] T(1-p,t,ty) \right),
\end{equation}
for $k \geq 1$ (see \cite[Equation (9)]{BCM}). We briefly recall here what this statment means and how to obtain it.

\begin{figure}[ht!]
\centering
\includegraphics[width=0.7\linewidth]{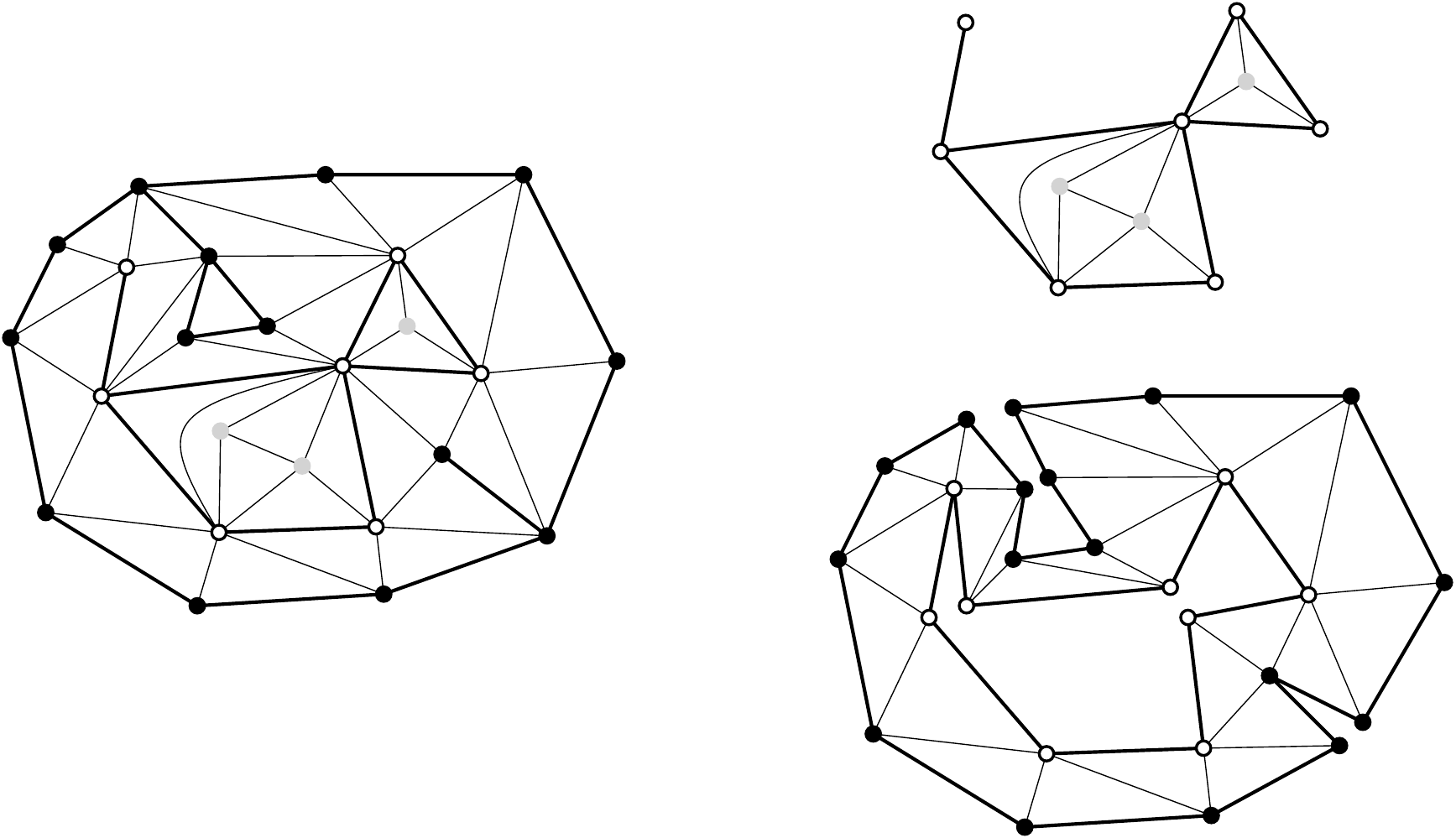}
\caption{\label{fig:gasket}Gasket decomposition: A face of the root cluster (black vertices in the figure) is filled with a triangulation with arbitrary white boundary (upper right figure, grey vertices can be either black or white) and a necklace (lower right figure) triangulating the region between the two.}
\end{figure}

For every percolated rooted planar triangulation $\mathfrak t$, we denote respectively by $\mathrm v_\circ (\mathfrak t)$ and  $\mathrm v_\bullet (\mathfrak t)$ the number of white (resp. black) vertices. Let $\mathcal {T}^{\mathrm{perc}}$ be the set of all percolated rooted triangulations with both end vertices of the root edge black. For every $\mathfrak t \in \mathcal {T}^{\mathrm{perc}}$, we denote by $\mathfrak C (\mathfrak t)$ the cluster of its root edge viewed as a planar rooted map (whose root is the root edge of the triangulation).
The fact that the root cluster is a random Boltzmann map comes from the following identity (see Equation (7) of \cite{BCM}): for every non atomic map $\mathfrak m$ one has
\begin{equation} \label{eq:MapMass}
\sum_{\mathfrak t \in \mathcal {T}^{\mathrm{perc}} \, : \, \mathfrak C (\mathfrak t) = \mathfrak m} t^{e(\mathfrak t)} p^{\mathrm v_\bullet (\mathfrak t)} (1-p)^{\mathrm v_\circ (\mathfrak t)}
= p^2 \,\cdot \prod_{f \in F(\mathfrak m)} q_{\mathrm{deg}(f)} (p,t).
\end{equation}
This identity stems from the classical gasket decomposition \cite{BBGa} sometimes called island decomposition \cite{BCM} and we briefly explain how it is obtained in the following lines (see Figure \ref{fig:gasket} for an illustration of these arguments). Fix $\mathfrak m$ a non atomic map and $\mathfrak t$ a site-percolated triangulation with root cluster $\mathfrak m$ colored black. For such a triangulation $\mathfrak t$, each face of its root cluster is filled with a triangulation with an arbitrary boundary of white vertices, and a \emph{necklace} of triangles with no additional vertices between this triangulation with a boundary and the cluster. For each cluster face of degree $k \geq 1$ filled with a triangulation with perimeter $l\geq 1$, there are $\binom{k+l-1}{k-1}$ different possible necklaces, and each of these necklaces require an additional $k+l$ edges. Taking into account the case where the cluster face has degree $3$ and can be part of the triangulation, this writes:
\begin{align*}
& \sum_{\mathfrak t \in \mathcal {T}^{\mathrm{perc}} \, : \, \mathfrak C (\mathfrak t) = \mathfrak m} t^{e(\mathfrak t)} p^{\mathrm v_\bullet (\mathfrak t)} (1-p)^{\mathrm v_\circ (\mathfrak t)}\\
& \quad = 
p^{|V(\mathfrak m)|} \, t^{|e(\mathfrak m)|} \prod_{f \in F(\mathfrak m)} \left(
\mathbf{1}_{\mathrm{deg}(f) = 3} + 
\sum_{l \geq 0} \binom{\mathrm{deg}(f)+l-1}{\mathrm{deg}(f)-1} t^{\mathrm{deg}(f)+l}  [y^l] T(1-p,t,y) \right),\\
&\quad = p^2 \cdot
\prod_{f \in F(\mathfrak m)}
\frac{1}{p}
(pt)^{\mathrm{deg}(f)/2}
\left(
\mathbf{1}_{\mathrm{deg}(f) = 3} + 
\sum_{l \geq 0} \binom{\mathrm{deg}(f)+l-1}{\mathrm{deg}(f)-1} t^{\mathrm{deg}(f)+l}  [y^l] T(1-p,t,y) \right),
\end{align*}
where we used Euler's relation $|V(\mathfrak m)| -2 = \sum_{f \in F(\mathfrak m)} \left( \mathrm{deg}(f)/2 - 1\right)$ in the last display. The expression \eqref{eq:MapMass} follows easily and we refer the reader to \cite[Section 2.2]{BCM} for additional details. Note for future reference that for any $k \geq 1$, the weight $q_k(p,t)$ is $1/p$ times the generating series of all triangulations of the $k$-gon with a weight $t$ per inner edge, a weight $\sqrt{pt}$ per boundary edge, and a weight $1-p$ per inner vertex adjacent to the boundary.

\bigskip

We can define the partition function of our percolated triangulations by:
\begin{equation} \label{eq:Zsum}
\mathcal Z (p,t) = \frac{1}{p^2} \cdot \sum_{\mathfrak t \in \mathcal {T}^{\mathrm{perc}}} t^{e(\mathfrak t)} p^{\mathrm v_\bullet (\mathfrak t)} (1-p)^{\mathrm v_\circ (\mathfrak t)}.
\end{equation}
From identity \eqref{eq:MapMass}, denoting by $\mathcal M$ the set of all non atomic rooted planar maps, we see that
\begin{align} \label{eq:partfun}
\mathcal Z (p,t) = \sum_{\mathfrak m \in \mathcal M} \prod_{f \in F(\mathfrak m)} q_{\mathrm{deg}(f)} (p,t).
\end{align}
Notice that this sum is finite when $p \in (0,1)$ and $t \in (0,t_c]$, meaning that the weight sequence $\mathbf q(p,t)$ is admissible in the sense of of \cite{MiermontInvariance}. We will need the asymptotic behavior of the coefficients in $t^3$ of the series $\mathcal Z(p,t)$:

\begin{prop} \label{prop:Ztcoef}
Fix $p \in (0,1)$, we have
\[
[t^{3n}] \mathcal Z(p,t) \underset{n\to\infty}{\sim}  \frac{\sqrt{2}\, \left(3 p-3+2 \sqrt{3}\right)}{2 p \sqrt{\pi}} \, t_c^{-3n} \, n^{-5/2}.
\]
\end{prop}
\begin{proof}
By opening the root edge of the triangulations appearing in the sum \eqref{eq:Zsum}, we can see that, for $p \in (0,1)$ and $|t| \leq t_c$, the partition function $\mathcal Z(p,t)$ is given by
\begin{equation}\label{eq:Z}
\mathcal Z (p,t) = \frac{1}{p^2 \, t} T_2(p,t) = \frac{1}{p^2 \, t^3}  \, t^2 T_2(p,t) = \frac{1}{t^3} tT_1(p,t) \left( 1 +\frac{1-p}{p^3} tT_1(p,t) \right).
\end{equation}
From Lemma \ref{lemm:parU}, The function $tT_1(t)$ seen as a series in $t^3$ has a unique dominant singularity at $t^3_c$ and we can obtain its asymptotic expansion at $t_c^3$ from the expansion of $U$. As a consequence, the function $\mathcal Z (p,t)$ also has a unique dominant singularity at $t_c^3$ ans we can easily obtain the following asymptotic expansion:
\begin{align*}
\mathcal Z(p,t) &= 
\frac{3 \sqrt{3}\, \left(7-4 \sqrt{3}\right) + p}{4 p}
-\frac{\sqrt{3}\, \left(3 p-27+16 \sqrt{3}\right)} {4 p} \, \left( 1- t^3/t_c^3 \right) \\
& \qquad + \frac{2 \sqrt{2}\, \left(3 p-3+2 \sqrt{3}\right)}{3 p} \, \left( 1- (t/t_c)^3)^{3/2} \right)
+ \mathcal O \left( \left( 1- t^3/t_c^3)^2 \right) \right).
\end{align*}
The asymptotic behavior of $[t^{3n}] \mathcal Z (p,t)$ then follows from the classical transfer Theorem \cite[Theorem VI.4]{FS}.
\end{proof}

\subsection{Generating series of the weights and of the associated Boltzmann maps}

The generating series of the weight sequence $\mathbf q (p,t)$ is straightforward to compute. As shown by the following lines, the expression we obtain is valid for every $p\in (0,1)$, $t$ in the greater domain of analycity of $U(t^3)$, and $z$ close enough to $0$. More precisely we need $\sqrt{p t^3} z \in \mathbb C \setminus [1 , + \infty )$ and $\left( 1 - \sqrt{p t^3} z\right)^{-1}$ must belong to the domain of analycity in $y$ of $T(t,1-p,ty)$ which was studied in the previous Section. Under these conditions for $p$, $t$ and $z$, the computation of the weight sequence generating function is as follows:
\begin{align}
F_{\mathbf q (p,t) }(z) :&= \sum_{k\geq 1} q_k(p,t) z^k, \notag\\
&= \frac{1}{p} (pt)^{3/2} z^3 + \frac{1}{p} \sum_{l \geq 0}
 [y^l] T(1-p,t,ty) \sum_{k \geq 1} \binom{k+l-1}{k-1} (p t^3)^{k/2}  z^k, \notag\\
&= \frac{1}{p} (pt)^{3/2} z^3 +  \frac{1}{p} \sqrt{p t^3} z \sum_{l \geq 0}
 [y^l] T(1-p,t,ty) \, \left( 1 - \sqrt{p t^3} z\right)^{-l-1}, \notag \\
&= \frac{1}{p} (pt)^{3/2} z^3 + \frac{1}{p} \frac{\sqrt{pt^3}z}{1-\sqrt{pt^3}z} T \left( 1-p,t,\frac{t}{1-\sqrt{pt^3}z} \right) \label{eq:genF}.
\end{align}

\bigskip

We will need expressions for the pointed and unpointed disk generating functions associated to the weight sequence $\mathbf q (p,t)$. For every $l \geq 1$, let $\mathcal M^l$ denote the set of all rooted planar maps with root face of degree $l$ (for $l=1$, this set contains only the atomic map). The unpointed disk generating function is defined as follows for $|z|$ large enough:
\[
W_{\mathbf q (p,t)}(z) := \sum_{l \geq 1} \, \left( \sum_{\mathfrak m \in \mathcal M^l} \prod_{f \in F(\mathfrak m) \setminus \{ f_r \}} q_{\mathrm{deg} (f)}(p,t) \right) \, z^{-l-1},
\]
where we denote the root face of a planar map by $f_r$. From our discussion establishing \eqref{eq:MapMass}, we can compute the coefficients of these series. Indeed, from Equation \eqref{eq:MapMass} and the fact that $q_l$ is $1/p$ times the generating series of triangulations with a boundary of perimeter $l$ counted with a weight  $t$ per inner edge, a weight $\sqrt{pt}$ per boundary edge, and a weight $1-p$ per inner vertex adjacent to the boundary, we have:
\begin{align*}
\sum_{\mathfrak m \in \mathcal M^l} \, \sum_{\mathfrak t \in \mathcal {T}^{\mathrm{perc}} \, : \, \mathfrak C (\mathfrak t) = \mathfrak m} t^{e(\mathfrak t)} p^{\mathrm v_\bullet (\mathfrak t)} (1-p)^{\mathrm v_\circ (\mathfrak t)}
& = 
p  \sqrt{p \,t}^{-l} \, q_l(p,t) \, \sum_{\mathfrak t \in \mathcal T_l}  t^{e(\mathfrak t)} p^{\mathrm v_{\mathrm{out}}(\mathfrak t)}.
\end{align*}
Comparing with the right hand side of \eqref{eq:MapMass}, this gives for every $l \geq 1$:
\begin{equation} \label{eq:Wbulletdef}
\sum_{\mathfrak m \in \mathcal M^l} \prod_{f \in F(\mathfrak m) \setminus \{ f_r \}} q_{\mathrm{deg} (f)}(p,t) 
= \frac{1}{p\,\sqrt{p t}^l} \, [y^l] T(p,t,y),
\end{equation}
and thus
\begin{equation} \label{eq:expW}
W_{\mathbf q (p,t) }(z) = \frac{1}{p\,z} \, T \left( p,t,\frac{t}{\sqrt{pt^3}z} \right).
\end{equation}

The pointed disk generating function is defined similarly:
\[
W_{\mathbf q (p,t)}^\bullet(z) = \sum_{l \geq 1} \, \left( \sum_{\mathfrak m \in \mathcal M^l} |V(\mathfrak m)| \, \prod_{f \in F(\mathfrak m) \setminus \{ f_r \}} q_{\mathrm{deg} (f)}(p,t) \right) \, z^{-l-1}.
\]
It has the following universal form: For $p\in(0,1)$ and and $t \in (0,t_c]$ fixed, there exists real numbers $c_+(p,t) >2$ and $c_-(p,t) \in (-c_+(p,t) , c_+(p,t))$ (the lower bound is excluded since our maps are not bipartite) such that for $z \in \mathbb C \setminus [c_-(p,t) , c_+(p,t)]$ one has  
\begin{equation} \label{eq:expWbullet}
W_{\mathbf q (p,t) }^\bullet (z) = \frac{1}{\sqrt{(z-c_+(p,t))(z-c_-(p,t))}}.
\end{equation}
This expression, sometimes called the one-cut Lemma, appears in several references. See for example \cite[Proposition 12]{BuddInfBoltz} or \cite[Section 6.1]{BBGb}. It is also established in these articles that the two disk generating functions $W$ and $W^\bullet$ have the same domain of analycity. Comparing our two expressions \eqref{eq:expW} and \eqref{eq:expWbullet}, we see that as a consequence
\begin{align} \label{eq:cpm}
c_\pm(p,t) = \frac{1}{\sqrt{p \, t^3} \,  y_\pm (p,t)}, 
\end{align}
where $ y_\pm (p,t)$ are the respective positive and negative singularities in $y$ of the series $T(p,t,ty)$ defined in Lemma \ref{lemm:parT}. Note that, as mentioned in the proof of Lemma \ref{lemm:parT}, this directly implies that $y_-(p,t) < - y_+(p,t)$.

\section{Percolation probability} \label{sec:main}

Recall that for any triangulation with a site percolation configuration, $\mathfrak C$ denotes the percolation cluster of its root vertex. Also recall that $\mathbb P^p_\infty$ denotes the law of the UIPT with vertices colored independently black with probability $p$ and white with probability $1-p$, conditioned on the event where the root edge has both end vertices colored black. We start by identifying the law of $\mathfrak C$ under $\mathbb P^p_\infty$ on the event where it is finite:

\begin{prop} \label{prop:probaCm}
For every $p \in (0,1)$ and every non atomic rooted finite map $\mathfrak m$ we have
\[
\mathbb P^p_\infty \left( \mathfrak C = \mathfrak m \right) = \left( \prod_{f \in F(\mathfrak m)} q_{\mathrm{deg} (f)}(p,t_c) \right) \cdot \sum_{f \in F(\mathfrak m)} \frac{(p t^3_c)^{\mathrm{deg}(f)/2} \delta_{\mathrm{deg} (f)} (p)}{q_{\mathrm{deg} (f)}(p,t_c)},
\]
where the numbers $\left( \delta_{k} (p) \right)_{k\geq 1}$ are defined in Lemma \ref{lem:serDelta}.
\end{prop}

\begin{proof}
For $p \in (0,1)$ and $n \geq 1$, we denote by $\mathbb P_n^p$ the law of a uniform triangulation of the sphere with $3n$ edges and vertices colored independently black with probability $p$ and white with probability $1-p$, conditioned on the event where both end vertices of the root edge are colored black. From equation \eqref{eq:partfun}, we can write that for every finite non atomic rooted map $\mathfrak m$ we have
\begin{equation} \label{eq:lawclustersizen}
\mathbb P_n^p \left( \mathfrak C = \mathfrak m \right) =
\frac{[t^{3n}] \prod_{f \in F(\mathfrak m )} q_{\mathrm{deg}(f)} (p ,t)}{[t^{3n}] \mathcal Z ( p ,t)}.
\end{equation}
Since the event $\{ \mathfrak C = \mathfrak m \}$ is continuous for the local topology, the limit as $n \to \infty$ of the previous display gives access to the annealed law of the percolation cluster of the root in the UIPT.

As can be seen from Proposition \ref{prop:Ztcoef} and Lemma \ref{lem:serDelta}, the asymptotic behaviors of $[t^{3n}] \mathcal Z (p ,t)$ and of $[t^{3n}] (p t^3)^{-k/2} q_{k} (p ,t)$ for every fixed $k \geq1$ are all of the form $\mathrm{Cst} \cdot t_c^{-3n} n^{-5/2}$. It is then very classical to establish the limit
\begin{align*}
\lim_{n \to \infty} \frac{[t^{3n}] \prod_{f \in F(\mathfrak m )} q_{\mathrm{deg}(f)} (p ,t)}{[t^{3n}] \mathcal Z ( p ,t)}
&=
\lim_{n \to \infty} \frac{[t^{3n}] (pt^3)^{\mathrm e(\mathfrak m)} \prod_{f \in F(\mathfrak m )} (p t^3)^{- \mathrm{deg}(f)/2} q_{\mathrm{deg}(f)} (p ,t)}{[t^{3n}] \mathcal Z ( p ,t)},\\
& = \sum_{f \in F(\mathfrak m )}  (p t^3_c)^{\mathrm{deg}(f)/2} \delta_{\mathrm{deg}(f)} (p) \, \prod_{f' \in F(\mathfrak m ),f' \neq f} q_{\mathrm{deg}(f)} (p ,t_c),
\end{align*}
proving the proposition.
\end{proof}

\bigskip

Our next result is an integral formula for the probability that the root cluster $\mathfrak C$ of the UIPT is infinite:

\begin{prop} \label{prop:integralProbab}
Recall the definition of the quantities $c_+(p,t_c)$ and $c_+(p,t_c)$ given in equations~\eqref{eq:expWbullet} and~\eqref{eq:cpm}. For every $p \in (0,1)$ one has
\begin{equation} \label{eq:ProbFiniteInt}
\begin{split}
&\mathbb P_\infty^p \left(|\mathfrak C| < \infty \right) \\
& =\frac{1}{2  \pi} \int_{c_-(p,t_c)}^{c_+(p,t_c)} \frac{dz}{z} \Delta \left( p,\sqrt{p t^3_c} \, z \right) \left(z + \frac{c_+(p,t_c) + c_-(p,t_c)}{2} \right) \sqrt{(c_+(p,t_c) - z)(z-c_-(p,t_c))},
\end{split}
\end{equation}
where the function $\Delta(p,z)$ is defined in Lemma~\ref{lem:serDelta}.
\end{prop}

\begin{proof}
We start with the formula established in Proposition \ref{prop:probaCm}. Summing over every finite non atomic map gives for every $p \in (0,1)$:
\[
\mathbb P_\infty^p \left(|\mathfrak C| < \infty \right)
=
\sum_{\mathfrak m \in \mathcal M} \, 
\sum_{f \in F(\mathfrak m )} (p t^3_c)^{\mathrm{deg}(f)/2} \delta_{\mathrm{deg}(f)} (p) \, \prod_{f' \in F(\mathfrak m ) \setminus \{f\}} q_{\mathrm{deg}(f')} (p ,t_c).
\]
By opening the root edges of the maps $\mathfrak m$ involved in the previous display, the sum transforms into a sum over maps in $\mathcal M^2$, the set of all rooted planar maps with root face of degree $2$:
\[
\mathbb P_\infty^p \left(|\mathfrak C| < \infty \right)
=
\sum_{\mathfrak m \in \mathcal M^2} \, 
\sum_{f \in F(\mathfrak m ) \setminus \{f_r\}} (p t^3_c)^{\mathrm{deg}(f)/2}\delta_{\mathrm{deg}(f)} (p) \, \prod_{f' \in F(\mathfrak m )\setminus \{f , f_r\}} q_{\mathrm{deg}(f')} (p ,t_c),
\]
where $f_r$ denotes the root face of a map. By rooting the non root face $f$ involved in the last sum (so that $f$ lies on the right hand side of this additional root) we can transform the sum into a sum over maps of the cylinder. For $l_1$, $l_2 \geq 1$, let $\mathcal M^{(l_1,l_2)}$ denote the set of all planar maps with two distinct marked rooted faces $f_1,\, f_2$ of respective degree $l_1,\, l_2$ (rooted so that the corresponding marked faces lie on the right hand side of their respective root). Taking into account the $\mathrm{deg} (f)$ possible roots for the face $f$ in the sum we have:
\begin{equation} \label{eq:probaCfiniteCyl}
\mathbb P_\infty^p \left(|\mathfrak C| < \infty \right)
=
\sum_{k \geq 1} \, \frac{1}{k} \,
\sum_{\mathfrak m \in \mathcal M^{(2,k)}} \, 
(p t^3_c)^{k/2} \delta_{k} (p) \, \prod_{f' \in F(\mathfrak m )\setminus \{f_1 , f_2\}} q_{\mathrm{deg}(f')} (p ,t_c).
\end{equation}
Setting for every $k \geq 0$
\[
\varphi_k(p) = \frac{1}{k} \,
\sum_{\mathfrak m \in \mathcal M^{(2,k)}} \, 
\prod_{f' \in F(\mathfrak m )\setminus \{f_1 , f_2\}} q_{\mathrm{deg}(f')} (p ,t_c),
\]
and denote the generating series of these numbers by
\[
\Phi(p,z) = \sum_{k \geq 1} \varphi_k(p) z^k,
\]
the sum in \eqref{eq:probaCfiniteCyl} takes the form of the Hadamard product of $\Delta(p, \sqrt{p t^3_c} z)$ and $\Phi (p, z)$ evaluated at $z=1$:
\[
\mathbb P_\infty^p \left(|\mathfrak C| < \infty \right)
=
\sum_{k \geq 1} (p t^3_c)^{k/2} \delta_k(p) \cdot \varphi_k(p)
= \Delta \left( p,\sqrt{p t^3_c} \, z \right) \odot \Phi (p,z) \vert_{z=1}.
\]
We will compute this Hadamard product with the help of its contour integral representation:
\[
\mathbb P_\infty^p \left(|\mathfrak C| < \infty \right)
=
\frac{1}{2i\pi}
\oint_\gamma \frac{dz}{z} \Delta \left( p,\sqrt{p t^3_c} \, z \right) \, \Phi(p,1/z),
\]
where the contour $\gamma$ must lie in a region enclosing $0$ where both functions $z \mapsto\Delta \left( p,\sqrt{p t^3_c} \, z \right)$ and $z \mapsto \Phi (p,1/z)$ are analytic (see for example \cite[Section VI.10.2]{FS}).

\bigskip

To compute this integral, we first have to compute $\Phi$. To this end, we introduce the cylinder generating functions of Boltzmann planar maps. It is defined for $|z_1|$ and $|z_2|$ large enough by
\[
W^{\mathrm{cyl}}_{\mathbf q(p,t)} (z_1,z_2) = \sum_{l_1,l_2 \geq 0}
\frac
{\sum_{\mathfrak m \in \mathcal M^{(l_1,l_2)}} \, 
\, \prod_{f' \in F(\mathfrak m )\setminus \{f_1 , f_2\}} q_{\mathrm{deg}(f')} (p ,t)}
{z_1^{l_1+1} \, z_2^{l_2+1}}.
\]
In a similar fashion than the pointed disk generating function $W^\bullet$ of Section \ref{sec:genser}, this series has a universal form involving the two constants $c_\pm(p,t)$:
\[
\begin{split}
& W^{\mathrm{cyl}}_{\mathbf q(p,t)} (z_1,z_2) \\
& = \frac{1}{2(z_1 - z_2)^2} \left(
W^\bullet_{\mathbf q(p,t)} (z_1) W^\bullet_{\mathbf q(p,t)} (z_2)
\left(z_1 z_2  - \frac{c_+(p,t) + c_-(p,t)}{2} (z_1 + z_2) + c_+(p,t) c_-(p,t) \right)
 -1 \right).
\end{split}
\]
This formula appears in \cite{BorotEG}, \cite{Eynard} and a multitude of other references. We refer to \cite{AM} for a proof and a review of the literature on this matter.

Of particular interest to us will be the generating function where the first root face has degree 2:
\begin{align*}
[z_1^{-3} ] & W^{\mathrm{cyl}}_{\mathbf q(p,t)} (z_1,z)\\
&= 
\sum_{k \geq 0} \, z^{-k-1} \, 
\sum_{\mathfrak m \in \mathcal M^{(2,k)}} \, 
\, \prod_{f' \in F(\mathfrak m )\setminus \{f_1 , f_2\}} q_{\mathrm{deg}(f')} (p ,t),\\
&= - z + W^\bullet_{\mathbf q(p,t)} (z)
\left( z^2 - \frac{c_+(p,t) + c_-(p,t)}{2} z +\frac{1}{4} c_+(p,t) c_-(p,t) -\frac{1}{8} \left( c_+^2(p,t) + c^2_-(p,t) \right) \right).
\end{align*}
The antiderivative of this function has a simple expression giving the identity:
\begin{align*}
\Phi (p,1/z) & =
\sum_{k \geq 0} \, \frac{1}{k}\, z^{-k} \, 
\sum_{\mathfrak m \in \mathcal M^{(2,k)}} \, 
\, \prod_{f' \in F(\mathfrak m )\setminus \{f_1 , f_2\}} q_{\mathrm{deg}(f')} (p ,t_c), \\
&= \frac{z^2}{2}
- \frac{z + \frac{c_+(p,t_c) + c_-(p,t_c)}{2}}{2} \sqrt{(z - c_+(p,t_c) ) (z - c_-(p,t_c) ) }.
\end{align*}

From this expression, we see that the function $z \mapsto \Phi(p,z)$ is analytic on $\mathbb C \setminus [c_-(p,t_c) , c_+(p,t_c)]$. From Lemma \ref{lem:serDelta}, we know that $z \mapsto \Delta(p,\sqrt{p t^3_c}z)$ is analytic on $\mathbb C \setminus \left[ c_+(p,t_c) , + \infty \right)$. We cannot directly pick an appropriate contour $\gamma$ to compute our Hadamard product of series evaluated at $1$, however, if we take $w \in (0,1)$ we have
\begin{align}
\Delta \left( p,\sqrt{p t^3_c} \, z \right) \odot \Phi (p,z) \vert_{z=w}=
\frac{1}{2i\pi}\oint_{\gamma(w)} \frac{dz}{z}\Delta \left( p, w \, \sqrt{p t^3_c} \, z \right) \, \Phi(p,1/z), \label{eq:HadamardContour}
\end{align}
where the contour $\gamma(w)$ encloses the interval $[c_-(p,t_c) , c_+(p,t_c)]$ and crosses the positive real line at some point inside the interval $(c_+(p,t_c), c_+(p,t_c) + \varepsilon)$ for some $\varepsilon >0$. Note that the in the last display, the function $z \mapsto \Delta (p,w \sqrt{pt_c^3} z) /z$ is well defined and continuous at $z=0$. Using the fact that $\Delta (p,w \, \sqrt{p t^3_c} \, z)$ is analytic inside $\gamma(w)$ and deforming the contour gives, for every $w \in (0,1)$,
\begin{align*}
\Delta & \left( p,\sqrt{p t^3_c} \, z \right) \odot \Phi (p,z) \vert_{z=w}\\
&= \frac{1}{2i\pi}
\oint_{\gamma(w)} \frac{dz}{z} \Delta \left( p,w \,\sqrt{p t^3_c} \, z \right) \, \left( \frac{z^2}{2}
- \frac{z + \frac{c_+(p,t_c) + c_-(p,t_c)}{2}}{2} \sqrt{(z - c_+(p,t_c) ) (z - c_-(p,t_c) )} \right),\\
& = -\frac{1}{4 i \pi} \oint_{\gamma(w)} \frac{dz}{z} \Delta \left( p,w \, \sqrt{p t^3_c} \, z \right)) \left(z + \frac{c_+(p,t_c) + c_-(p,t_c)}{2} \right) \sqrt{(z-c_+(p,t_c))(z-c_-(p,t_c))},\\
& = \frac{1}{2  \pi} \int_{c_-(p,t_c)}^{c_+(p,t_c)} \frac{dz}{z} \Delta \left( p,w \, \sqrt{p t^3_c} \, z \right) \left(z + \frac{c_+(p,t_c) + c_-(p,t_c)}{2} \right) \sqrt{(c_+(p,t_c) - z)(z-c_-(p,t_c))}.
\end{align*}
Taking the limit as $w \to 1$ finaly gives the proposition
\end{proof}

\bigskip

We are now ready to prove our first main result:

\begin{proof}[Proof of Theorem \ref{th:offcrit}]

The proof basically consists on computing the integral \eqref{eq:ProbFiniteInt}, which still requires some work.
The change of variables $z = (1 - 1/y) (p t_c^3)^{-1/2}$ gives:
\begin{align*}
&\mathbb P_\infty^{\nu} (|\mathfrak C| < \infty)\\
& =
\frac{1}{2 \pi \, p \, t_c^3} \int_{\frac{y_-(p,t_c)}{y_-(p,t_c)-1}}^{\frac{y_+(p,t_c)}{y_+(p,t_c)-1}} \, \frac{dy}{y(y-1)} \,
\hat \Delta \left( p,V(1-p,U(t_c^3),y) \right) \,
\left(\frac{y -1}{y} + \frac{1}{2y_-(p,t_c)} + \frac{1}{2y_+(p,t_c)} \right) \\
& \qquad \qquad  \qquad \times
\sqrt{
\left( \frac{1}{y_+(p,t_c)} - \frac{y -1}{y} \right)
\left( \frac{y -1}{y}- \frac{1}{y_-(p,t_c)} \right)
},
\end{align*}
where $\hat \Delta$ is defined in Lemma \ref{lem:serDelta} and $V(1-p,U,y)$ is defined in Lemma \ref{lemm:parT}.
To calculate this new integral, we want to do the change of variables $y= \hat y(1-p,U(t_c^3),2\sqrt{3}/3 - V)$, where $\hat y$ is defined in Lemma \ref{lemm:parT} (as we will see shortly, this change of variables instead of simply $y= \hat y(1-p,U(t_c^3),V)$ simplifies calculations a bit more).

Recall that, from Lemma \ref{lemm:parT} and its proof, the function $V \mapsto \hat y (1-p,U(t_c^3),V)$ is an increasing bijection from $[V_-(1-p,U(t_c^3)), V_+(1-p,U(t_c^3))]$ onto $[y_-(1-p,t_c), y_+(1-p,t_c)]$, and is analytic on $(V_-(1-p,U(t_c^3)), V_+(1-p,U(t_c^3)))$. In view of this, for our change of variables, we want $2\sqrt{3}/3 - V$ to be in a sub-interval of $[V_-(1-p,U(t_c^3)), V_+(1-p,U(t_c^3))]$. This will also enable us to use the rational parametrization of the series $\Delta$ of Lemma \ref{lem:serDelta}. The details of the calculations that follow are available in the Maple file \cite{Maple}.

We have to solve for $V$ the two equations
\[
\hat y\left( 1-p, U(t_c^3),\frac{2\sqrt 3}{3} - V \right) = \frac{y_\pm(p,t_c)}{y_\pm(p,t_c) - 1}.
\]
There is an interesting symmetry to exploit: for every $p \in (0,1)$ and every $V \in \mathbb C$ one has
\[
\hat y \left( p,U(t_c^3),V \right) = \frac{\hat y(1-p,U(t_c^3),\frac{2\sqrt 3}{3} - V)}{ \hat y(1-p,U(t_c^3),\frac{2\sqrt 3}{3} - V) - 1},
\]
therefore, we want to solve for $V$ the two equations
\begin{equation} \label{eq:Vint}
\hat y\left( p, U(t_c^3),V \right) = y_\pm(p,t_c).
\end{equation}

Since $\hat y\left( p, U(t_c^3),V \right)$ is a rational fraction in $V$ of degree $2$ for its numerator and $3$ for its denominator, the solutions of equation \eqref{eq:Vint} are the roots of a polynomial of degree $3$. By definition, one of the root of this polynomial is $V_\pm(p,U(t_c^3))$, and it will even be a double root except for $V_-(p,U(t_c^3))$ when it is a negative pole of $\hat y$. Fortunately, we can compute explicitly these values. Indeed, the stationary points of $V \mapsto \hat y(p,U(t_c^3),V)$ are the roots of the polynomial
\[
\left(9 V^3 - 9V^2 \sqrt 3 + 4p \sqrt 3 \right) \, \left(V - \sqrt 3 /3 \right).
\]
The roots of the polynomial of degree $3$ are given by
\begin{align*}
V_m(p) & = - \frac{\sqrt 3}{3} \, \frac{\sqrt p}{\cos \left( \frac{\mathrm{arccos} (\sqrt{p})}{3}\right)} <0,\\
V_l(p) & = \frac{\sqrt 3}{3} \, \frac{\sqrt p}{\cos \left( \frac{\mathrm{arccos} (\sqrt{p})}{3} - \frac{\pi}{3}\right)} \in [0,2\sqrt{3}/3],\\
V_r(p) & = \frac{2\sqrt 3}{3} - V_m(p) \geq 2\sqrt{3}/3.
\end{align*}
From there, we see that $V_-(p,U(t_c^3)) = V_m(p)$ when $\hat y$ has no negative pole, which is the case when $p > \frac{1}{2} - \frac{5 \sqrt{3}}{18} \sim 0.018$. Since we are interested in $p \geq 1/2$, the case when $\hat y$ has negative poles will not bother us. The positive singularity is given by
\[
V_+(p,U(t_c^3)) =
\begin{cases}
\frac{\sqrt{3}}{3} & \text{when $p \geq 1/2$},\\
V_l(p) & \text{when $p \leq 1/2$}.
\end{cases}
\]

We can compute the third root of equation \eqref{eq:Vint} from the constant term of the polynomial, which is always  $-\frac{2\sqrt{3} p}{9}$. When $p>\frac{1}{2} - \frac{5 \sqrt{3}}{18}$, this root is given by
\[V^i_\pm (p) = \frac{2\sqrt{3} p}{9 V_\pm(p,U(t_c^3))^2}.\]
Since $V_\pm(p,U(t_c^3)) \leq \sqrt{3}/3$, we can see that $2 \sqrt{3}/3 - V_\pm(p,U(t_c^3))$ is always larger than $\sqrt{3}/3$, and therefore $V_+(1-p,U(t_c^3))$. From the explicit expression of $V_\pm(p,U(t_c^3))$, we can also check that $2\sqrt{3}/3 - V_\pm^i(p)$ is always between $0$ and $V_+(1-p,U(t_c^3))$. The correct bounds for our change of variables are then $V^i_\pm (p)$ since they are the only solutions inside the interval $[V_-(1-p,U(t_c^3)),V_+(1-p,U(t_c^3))]$. Note that we have $0<V_+^i(p) < V_-^i(p)$. This discussion also provides the following nice factorizations:
\begin{align*}
 \frac{\hat y (1-p, U(t_c^3), \frac{2 \sqrt 3}{2} -V)) -1}{\hat y (1-p, U(t_c^3), \frac{2 \sqrt 3}{2} -V))}- \frac{1}{y_\pm(p,t_c)}
&=
\frac{1}{\hat y (p, U(t_c^3), V))}- \frac{1}{y_\pm(p,t_c)}, \\
&=  \frac{\left(V - V_\pm^i(p) \right) \left(V - V_\pm(p,U(t_c^3)) \right)^2}{2 \, V \left(V-\frac{2\sqrt{3}}{3} \right)}.
\end{align*}
With the explicit expression for $\hat \Delta$ in terms of $V$ given in Lemma \ref{lem:serDelta}, it is easy to check:
\[
\hat \Delta \left( p,V \right) \,
\frac{\partial_V \hat y(1-p,U(t_c^3),V)}{\hat y(1-p,U(t_c^3),V) \, (\hat y(1-p,U(t_c^3),V) - 1)} = \frac{1}{3 \, (2\sqrt{3} -3(1-p)) \left( V - \frac{\sqrt{3}}{3} \right)^2}.
\]
Using the change of variables $y = \hat y(1-p,U(t_c^3),2\sqrt{3}/3 - V)$ we finally get:
\begin{align*}
&\mathbb P_\infty^{\nu} (|\mathfrak C| < \infty) \\
& =
\frac{1}{12 \,p \, t_c^3 \, (2\sqrt{3} -3(1-p)) \, \pi}
\int_{V_+^i(p)}^{V_-^i(p)} dV \,\sqrt{(V_-^i(p)-V)(V-V_+^i(p))}\\
& \qquad \qquad \qquad \qquad \qquad \qquad \qquad \qquad \qquad  \times \frac{(V-V_+(p,U(t_c^3)))(V-V_-(p,U(t_c^3)))}{V \, \left(\frac{2\sqrt{3}}{3} -V\right) \, \left( V - \frac{\sqrt{3}}{3} \right)^2}\\
& \qquad \qquad \qquad \qquad \qquad \qquad \qquad \qquad \qquad  \times \left( \frac{1}{\hat y(p,U(t_c^3),V)} + \frac{1}{2} \left(\frac{1}{y_+(p,t_c)} + \frac{1}{y_-(p,t_c)} \right) \right).
\end{align*}

Embarrassingly, we were not able to simplify directly this integral (it is equal to $1$ when $p<1/2$ since we know that the critical percolation threshold is $1/2$ !). Fortunately, when $p > 1/2$, Maple is able to give an expression in terms of $V_\pm(p,U(t_c^3))$.  Using basic trigonometric identities, we can further simplify this expression into the one given in the theorem after injecting the expressions for $V_\pm(p,U(t_c^3))$.
\end{proof}

\section{Cluster volume and perimeter} 

\subsection{Admissibility equations and volume-modified weight sequence} \label{sec:criteq}

We review in this Section additional background on Boltzmann planar maps that we will use in the proof of Theorem \ref{th:volume}. We refer the reader to the references \cite{BCM,BuddInfBoltz,MiermontInvariance} for details.

\bigskip

For $p \in (0,1)$ and $t \in (0,t_c]$, consider the two following bivariate power series in $(z_1,z_2)$:
\begin{align*}
f^\bullet (p,t;z_1,z_2) &= \sum_{k,k' \geq 0} z_1^k z_2^{k'} \binom{2k+k' +1}{k+1,k,k'} q_{2+2k+k'} (p,t), \\
f^\diamond (p,t;z_1,z_2) &= \sum_{k,k' \geq 0} z_1^k z_2^{k'} \binom{2k+k'}{k,k,k'} q_{1+2k+k'} (p,t).
\end{align*}
These two functions are linked with Boltzmann maps with weight sequence $(q_k(p,t_c))_{k\geq 1}$ by the Bouttier-Di Francesco-Guitter bijection \cite{BDFG} and will be instrumental in the remaining of this work. In our case, we can compute alternative expressions for these functions that will be more amenable to analysis, see Lemma \ref{lem:BDFGintegral}.

Recall the parameters $c_+(p,t)$ and $c_-(p,t)$ defined in Equation \eqref{eq:cpm}. Since the weight sequence $\mathbf q (p,t)$ is admissible, the two functions defined above are well defined at least in the domain $|z_1| \leq z^+(p,t)$ and $|z_2| \leq z^\diamond (p,t)$ where $z^+(p,t)$ and $z^\diamond (p,t)$ are positive real numbers defined by
\begin{equation} \label{eq:zpd}
c_\pm(p,t) = z^\diamond(p,t) \pm 2 \sqrt{z^+(p,t)}.
\end{equation}
Furthermore, from Proposition 4.2 and Lemma 4.4 of \cite{BCM}, $(z^+(p,t),z^\diamond(p,t))$ is the minimal solution of the system of equations
\begin{equation} \label{eq:system}
\begin{cases}
f^\bullet(p,t;z^+(p,t),z^\diamond(p,t)) &= 1 - \frac{1}{z^+(p,t))},\\
f^\diamond (p,t;z^+(p,t),z^\diamond(p,t)) &= z^\diamond(p,t).
\end{cases}
\end{equation}
In addition, when $p=1/2$ and $t=t_c$, the weight sequence is critical (see Theorem 1.1 of \cite{BCM}), and $(z^+(1/2,t_c),z^\diamond(1/2,t_c))$ is the unique solution of the the system of equations \eqref{eq:system} such that
\begin{equation}\label{eq:critdiff}
\left( \partial_{z_2} + \sqrt{z_1} \partial_{z_1} \right) f^\diamond (p,t;z^+(p,t),z^\diamond(p,t)) = 1.
\end{equation}

\bigskip

For $g \in (0,1]$, let $(z^+(p,t;g) , z^\diamond (p,t;g) )$ be the unique solution in $(0,z^+(p,t)] \times (0,z^\diamond (p,t)]$ of the system of equations
\begin{equation} \label{eq:systemg}
\begin{cases}
f^\bullet(p,t;z^+(p,t;g),z^\diamond(p,t;g)) &= 1 - \frac{g}{z^+(p,t;g))},\\
f^\diamond (p,t;z^+(p,t;g),z^\diamond(p,t;g)) &= z^\diamond(p,t;g).
\end{cases}
\end{equation}
Define
\[
c_\pm(p,t;g) = \frac{1}{g} \left(z^\diamond (p,t;g) \pm 2  \sqrt{z^+ (p,t;g)} \right).
\]
From e.g. \cite[Equations (4.3) and (4.4)]{BCM}, the pointed disk generating function of Boltzmann maps with modified weight sequence $\mathbf q (p,t;g) := \left( g^{(k -2)/2} \, q_{k} (p ,t) \right)_{k \geq 1}$ defined as in \eqref{eq:Wbulletdef} is given by
\begin{equation}
W_{\mathbf q (p,t;g)}^\bullet (z) = \frac{1}{\sqrt{ (z- c_+(p,t;g)) (z-c_-(p,t;g) - z)}}.
\end{equation}

\subsection{Proof of Theorem \ref{th:volume}} \label{sec:vol}

We start with the perimeter exponent as it is the easiest of the two. From Equations \eqref{eq:MapMass} and \eqref{eq:Wbulletdef}, we can write
\[
\mathbb P_n ^p \left( |V(\partial \mathfrak C ) | = k \right) = \frac{[t^{3n}] \left(q_k(p,t) \cdot [z^{-(k+1)}] W_{\mathbf q (p,t)}(z) \right)}{[t^{3n}] \mathcal Z (p,t)}
= \frac{1}{p} \frac{[t^{3n}] \left(\sqrt{pt^3}^{-k} q_k(p,t) \cdot [y^k] T(p,t,ty) \right)}{[t^{3n}] \mathcal Z (p,t)}.
\]
Applying Lemmas \ref{lem:serDelta} and \ref{lem:Tkasymp} directly give the limit, which is the probability that $|V(\partial \mathfrak C ) | = k$ in the UIPT by continuity for the local topology. 
\[
\mathbb P_\infty ^p \left( |V(\partial \mathfrak C ) | = k \right) 
= \frac{1}{p} 
\left(
\delta_k(p) \, t_c^k T_k(p,t_c) + \sqrt{pt_c^3}^{-k} q_k(p,t_c) \, \theta_k(p)
\right).
\]
Using the asymptotics \eqref{eq:Ttcritexp}, \eqref{eq:qkcritexp}, \eqref{eq:deltakexp}  and \eqref{eq:thetakexp} derived in Section \ref{sec:expansionscrit} gives that for $p=1/2$,
\[
\mathbb P_\infty ^{1/2} \left( |V(\partial \mathfrak C ) | = k \right) 
\underset{k \to \infty}{\sim}
\kappa' \, k^{-4/3} 
\]
with
\begin{equation} \label{eq:kappaprime}
\kappa' = - 8 \, \frac{1}{\Gamma(4/3)} \left( \frac{8 \,3^{\frac{5}{6}}}{351}+\frac{2 \,3^{\frac{1}{3}}}{117} \right) \, \frac{3^{5/6}}{2 \Gamma(-2/3)} \simeq
0.454,
\end{equation}
proving the second statement of Theorem \ref{th:volume}.

\bigskip

We now turn on the first statement of the Theorem on the volume of the root cluster.
Although we are only interested in the case $p=1/2$, we start the proof with a generic $p \in (0,1)$ as it will be easier to follow. From proposition \ref{prop:probaCm} we have for every $g \leq 1$
\[
\mathbb E_\infty^p \left[ g^{|V(\mathfrak C)|} \right]
=
\sum_{\mathfrak m \in \mathcal M} \, g^{|V(\mathfrak m)|} \,
\sum_{f \in F(\mathfrak m )} (p t^3_c)^{\mathrm{deg}(f)/2} \delta_{\mathrm{deg}(f)} (p) \, \prod_{f' \in F(\mathfrak m ) \setminus \{f\}} q_{\mathrm{deg}(f')} (p ,t_c).
\]
Applying Euler's formula then gives
\[
\mathbb E_\infty^p \left[ g^{|V(\mathfrak C)|} \right]
= g \,
\sum_{\mathfrak m \in \mathcal M}  \,
\sum_{f \in F(\mathfrak m )} (g \, p t^3_c)^{\mathrm{deg}(f)/2} \delta_{\mathrm{deg}(f)} (p) \, \prod_{f' \in F(\mathfrak m ) \setminus \{f\}} g^{(\mathrm{deg}(f') -2)/2} \, q_{\mathrm{deg}(f')} (p ,t_c).
\]
Using the exact same line of reasoning as in the proof of Theorem \ref{th:offcrit}, we can compute this sum as the Hadamard product evaluated at $z=1$ of the functions $z  \mapsto \Delta(p, \sqrt{g p t_c^3})$ and the function $z \mapsto \Phi (p,z)$ where $c_+(p,t_c)$ and $c_-(p,t_c)$ are replaced by the two constants $c_+(p,t_c;g)$ and $c_-(p,t_c;g)$ associated to the pointed disk generating function $\left( (z- c_+(p,t_c;g)) (z-c_-(p,t_c;g) - z) \right)^{-1/2}$ of Boltzmann maps with weight sequence $\mathbf q (p,t_c;g) = \left( g^{(k -2)/2} \, q_{k} (p ,t_c) \right)_{k \geq 1}$ introduced in Section \ref{sec:criteq}. Indeed, the first function should be obvious, and the second comes from the antiderivative of cylinder generating functions associated to the weight sequence $\mathbf q (p,t_c;g)$ instead of the weight sequence $\mathbf q (p,t_c;1)$, which has the same universal form. Of course, it remains to calculate these two constants, which are less explicit than their counterparts $c_+(p,t_c)$ and $c_-(p,t_c)$. Nevertheless, after mimicking the part of the proof of Theorem \ref{th:offcrit} leading to \eqref{eq:ProbFiniteInt}, we arrive at the formula
\begin{align}
\mathbb E_\infty^p \left[ g^{|V(\mathfrak C)|} \right]
= 
\frac{g}{2  \pi} \int_{c_-(p,t_c;g)}^{c_+(p,t_c;g)} \frac{dz}{z} & \Delta \left( p, \sqrt{g \, p t^3_c} \, z \right) \left(z + \frac{c_+(p,t_c;g) + c_-(p,t_c;g)}{2} \right) \notag \\
& \times \sqrt{(c_+(p,t_c;g) - z)(z-c_-(p,t_c;g))}. \label{eq:Volgint}
\end{align}

\bigskip

We want to study the asymptotic behavior of the integral \eqref{eq:Volgint} when $g \to 1^-$. To do so, we must first study the dependency in $g$ of the two constants $c_\pm(p,t_c;g)$. We do not have a simple formula for the unpointed or pointed disk generating function of the Boltzmann maps as was the previously the case. However, the two constants can be studied via the solution the system of equations \eqref{eq:systemg}. Let us denote $z^+(g) =z^+(1/2,t_c;g)$ and $z^\diamond (g) =z^\diamond (1/2,t_c;g)$ the solution of this system. We can calculate an expansion as $g \to 1^-$ of these two quantities with Lemma \ref{lem:BDFGexpansion}. Indeed, the development of $f^\diamond$ gives
\begin{align*}
z^+ - z^+(g) &=  \sqrt{z^+} \, \left( z^\diamond - z^\diamond(g) \right)
+ \frac{\kappa^\diamond}{\partial_{z_1} f^\diamond (z^+,z^\diamond)} \,
\left( (z^+ - z^+(g)) + \sqrt{z^+} (z^\diamond - z^\diamond(g)) \right)^{7/6} \\
& \qquad +  o \left( \left((z^+ - z^+(g)) + \sqrt{z^+} (z^\diamond - z^\diamond(g)) \right)^{7/6} \right).
\end{align*}
Therefore we have
\begin{align} \label{eq:zdinter}
z^\diamond - z^\diamond(g) &=  \frac{1}{\sqrt{z^+}} \, \left( z^+ - z^+(g) \right) \notag \\
& \qquad - \frac{\kappa^\diamond}{\sqrt{z^+} \, \partial_{z_1} f^\diamond (z^+,z^\diamond)} \,
\left(2(z^+ - z^+(g)) \right)^{7/6} +  o \left( \left( z^+ - z^+(g) \right)^{7/6} \right).
\end{align}
Now the development of $f^\bullet$ gives
\begin{align*}
\frac{1}{z^+} - \frac{g}{z^+(g)} &= - \frac{1}{(z^+)^2} \left( z^+ - z^+(g) \right) \\
& \qquad + \left( \frac{\kappa^\diamond}{\sqrt{z^+}} + \kappa^\bullet \right) \left(2(z^+ - z^+(g)) \right)^{7/6} + o \left( \left(2(z^+ - z^+(g)) \right)^{7/6} \right).
\end{align*}
Expanding the left hand side of this identity then gives
\begin{align*}
\frac{1-g}{z^+} + \frac{1-g}{(z^+)^2}  \left( z^+ - z^+(g) \right) =  \left( \frac{\kappa^\diamond}{\sqrt{z^+}} + \kappa^\bullet \right) \left(2(z^+ - z^+(g)) \right)^{7/6} + o \left( \left(2(z^+ - z^+(g)) \right)^{7/6} \right),
\end{align*}
which in turns gives
\begin{align*}
z^+(g) = z^+ - \frac{1}{2} \, \left(\kappa^\diamond \sqrt{z^+} + \kappa^\bullet z^+ \right)^{-6/7} \left( 1- g \right)^{6/7}
+ o \left(  \left( 1- g \right)^{6/7} \right).
\end{align*}
Plugging this expression in equation \eqref{eq:zdinter} gives
\begin{align*}
z^\diamond(g) = z^\diamond - \frac{1}{2 \sqrt{z^+}} \, \left(\kappa^\diamond \sqrt{z^+} + \kappa^\bullet z^+ \right)^{-6/7} \left( 1- g \right)^{6/7}
+ o \left(  \left( 1- g \right)^{6/7} \right).
\end{align*}
We finally get an expansion for $c_\pm(g)$:
\begin{align*}
c_+(g) & = z^\diamond(g) + 2  \sqrt{z^+(g)} = c_+(1) - \frac{1}{\sqrt{z^+}} \, \left(\kappa^\diamond \sqrt{z^+} + \kappa^\bullet z^+ \right)^{-6/7} \left( 1- g \right)^{6/7}
+ o \left(  \left( 1- g \right)^{6/7} \right),\\
c_-(g) & = z^\diamond(g) - 2  \sqrt{z^+(g)} = c_-(1) + o \left(  \left( 1- g \right)^{6/7} \right).
\end{align*}

\bigskip

The change of variable $z = \phi (g,\xi ) := c_-(g) + \left( c_+(g) - c_-(g) \right) \xi$ in Equation \eqref{eq:Volgint} gives
\begin{align*}
\mathbb E_\infty^p \left[ g^{|V(\mathfrak C)|} \right]
&= 
\frac{g (c_+(g) - c_-(g))^2}{2  \pi} \int_{0}^{1} d\xi \sqrt{\xi (1-\xi)} \frac{\Delta \left( 1/2, \sqrt{g \, t^3_c/2} \, \phi(g,\xi) \right)}{\phi(g,\xi)} \left(\phi(g,\xi) + \frac{c_+(g) + c_-(g)}{2} \right), \\
& =
\frac{(c_+(1) - c_-(1))^2}{2  \pi} \int_{0}^{1} d\xi \sqrt{\xi (1-\xi)} \frac{\Delta \left( 1/2, \sqrt{g \, t^3_c/2} \, \phi(g,\xi) \right)}{\sqrt{g} \, \phi(g,\xi)} \left(\phi(1,\xi) + \frac{c_+(1) + c_-(1)}{2} \right) \\
& \qquad \qquad + \mathcal O \left( (1-g)^{6/7} \right)
\end{align*}
where we used the developments of $c_\pm(g)$ to obtain the second equality. In a similar fashion than in the proof of Lemma \ref{lem:BDFGexpansion}, the dominant singular term of the integral comes for the singularity at $g=1$ and $\xi =1$ of the term $\frac{\Delta \left( 1/2, \sqrt{g \, t^3_c/2} \, \phi(g,\xi) \right)}{\sqrt{g} \, \phi(g,\xi)}$. The asymptotic expansion \eqref{eq:deltacritexp} established in Section \ref{sec:expansionscrit} gives
\[
\frac{\Delta \left( 1/2 ,  z \right)}{z} \underset{z \to 1/2}{\sim} \left( \frac{16 \,3^{\frac{5}{6}}}{351}+\frac{4 \,3^{\frac{1}{3}}}{117} \right) \, \left( 1 - 2 z \right)^{-4/3}.
\]
Applying the same techniques as in the proof of Lemma \ref{lem:BDFGexpansion}, we see that the main singular term of the expansion of $\mathbb E_\infty^p \left[ g^{|V(\mathfrak C)|} \right]$ comes from the integral
\begin{align*}
\frac{\tilde \kappa}{\pi} \int_{0}^{1} d\xi \sqrt{\xi (1-\xi)} & \left( 1-2 \sqrt{g \, t^3_c/2} \, \phi(g,\xi) \right)^{-4/3} \\
&= \frac{\tilde \kappa}{8 \left(1 - 2 \sqrt{g \, t^3_c/2} c_-(g) \right)^{4/3}} \, _2F_1 \left(\frac{4}{3}, \frac{3}{2} ; 3 ; \frac{2 \sqrt{g \, t^3_c/2} \left( c_+(g) - c_-(g)\right)}{1 - 2 \sqrt{g \, t^3_c/2} c_-(g)} \right) ,
\end{align*}
with
\[
\tilde \kappa = \frac{(c_+(1) - c_-(1))^2  \left(\phi(1,1) + \frac{c_+(1) + c_-(1)}{2} \right) \sqrt{t_c^3/2}}{2} \left( \frac{4 \,3^{\frac{5}{6}}}{351}+\frac{3^{\frac{1}{3}}}{117} \right).
\]
Using the singular expansion of the hypergeometric function at $1$ and the developments of $c_\pm(g)$, we finally get
\begin{align*}
\mathbb E_\infty^p \left[ g^{|V(\mathfrak C)|} \right]
= 1 - &\frac{\tilde \kappa}{8 \left(1 - 2 \sqrt{\, t^3_c/2} c_-(1) \right)^{4/3}} \,
\frac{36 \sqrt{3} \,  \Gamma \! \left(\frac{5}{6}\right) \, \Gamma \! \left(\frac{2}{3}\right) }{\pi^{\frac{3}{2}}} \\
& \qquad \left( \frac{2 \sqrt{t_c^3/2}}{1 - 2 \sqrt{t_c^3/2} c_-(1)} 
\frac{1}{\sqrt{z^+}} \, \left(\kappa^\diamond \sqrt{z^+} + \kappa^\bullet z^+ \right)^{-6/7} \left( 1- g \right)^{6/7}
\right)^{1/6}
+ o \left( 1- g \right)^{1/7}.
\end{align*}

\bigskip

This expansion is unfortunately not enough to extract the asymptotic behavior of the probabilities $\mathbb P _\infty^{1/2} \left( | V(\mathfrak C) | = n\right) $ as $n \to \infty$. Indeed, the generating series of these probabilities is $\mathbb E_\infty^p \left[ g^{|V(\mathfrak C)|} \right]$, but this function could have singularities of modulus $1$ other than $1$ contributing to the asymptotic. However, we do not have this problem for tail probabilities. For every $n$, denote $p_n = \mathbb P _\infty^{1/2} \left( | V(\mathfrak C) | \geq n\right) $. A simple computation gives
\[
\sum_{n \geq 0} p_n g^n = \frac{1-g \, \mathbb E_\infty^p \left[ g^{|V(\mathfrak C)|} \right] }{1-g} \underset{g \to 1^-}{\sim}
 \kappa_1 \, \left( 1- g \right)^{-6/7},
\]
with
\[
\kappa_1 = \frac{\tilde \kappa}{8 \left(1 - 2 \sqrt{\, t^3_c/2} c_-(1) \right)^{4/3}} \,
\frac{36 \sqrt{3} \,  \Gamma \! \left(\frac{5}{6}\right) \, \Gamma \! \left(\frac{2}{3}\right) }{\pi^{\frac{3}{2}}}
\left( \frac{2 \sqrt{t_c^3/2}}{1 - 2 \sqrt{t_c^3/2} c_-(1)} 
\frac{1}{\sqrt{z^+}} \right)^{1/6}
\, \left(\kappa^\diamond \sqrt{z^+} + \kappa^\bullet z^+ \right)^{-8/7}.
\]
From there, a classical Tauberian Theorem (see e.g. Theorem VI.13 of \cite{FS} and the following discussion) establishes the asymptotic behavior of $p_n$ given in the Theorem with 
\begin{equation} \label{eq:kappa}
\kappa = \frac{\kappa_1}{ \Gamma(8/7)}
= \frac{63 \left(3^{\frac{1}{3}}+\frac{4 \,3^{\frac{5}{6}}}{3}\right) 3^{\frac{17}{21}} 7^{\frac{1}{7}} 137^{\frac{6}{7}} \Gamma \! \left(\frac{2}{3}\right)^{\frac{18}{7}} 2^{\frac{3}{14}} 5^{\frac{13}{14}}}{56992 \pi^{\frac{12}{7}} \Gamma \! \left(\frac{1}{7}\right)}
\simeq 0.278.
\end{equation}

\section{Technical Lemmas} 

\subsection{Dependency in \texorpdfstring{$t$}{t} of the weights} \label{sec:teclem}

\begin{lemm} \label{lem:serDelta}
Fix $p \in (0,1)$ and let $V_c^z$ be the power series in $z$ defined be $V_c^z = V(1-p,U(t_c^3),1/(1-z))$, where $V$ is defined in Lemma \ref{lemm:parT}. Let $\Delta(p,z)$ be the power series in $z$ defined by:
\begin{align}{}
&\Delta(p,z) \notag\\
&=  \hat \Delta (p,V_c^z) :=
3 \,
\frac
{V_c^z \, (2 \sqrt 3 -3V_c^z) \, \left(  9 (V_c^z)^3  - 9 (\sqrt 3 + 1) (V_c^z)^2 + 3(3 + 2\sqrt 3) V_c^z -  2 (1-p) \sqrt 3 \right)}
{(3(p - 1) + 2 \sqrt 3)\, \left( \sqrt 3 -3V_c^z \right)^3 \, \left(9 (V_c^z)^3 - 9 (V_c^z)^2 \sqrt 3 + 4 (1-p) \sqrt 3 \right)}.
\end{align}
For every $k \geq 1$ one has
\begin{equation} \label{eq:asymptot3nqk}
\frac{[t^{3n}] q_{k} (p ,t) \sqrt{p t^3}^{-k}}{[t^{3n}] \mathcal Z ( p ,t)}
\underset{n \to \infty}{\rightarrow} \delta_k(p)
\end{equation}
where the generating series of the numbers $\delta_k(p)$ is given by $\Delta(p,z)$, which is analytic in the domain $\mathbb C \setminus \left[ 1 - 1/y_+(p,t_c) , + \infty \right)$.
\end{lemm}
\begin{proof}
In the whole proof, $p\in (0,1)$ is fixed. All calculations are available in the Maple companion file \cite{Maple}. We start by proving that, for every $k \geq 1$, the series
\[\tilde q_k(p,t) = \sqrt{p t^3}^{-k} \cdot \left(q_k(p,t) - (p t)^{3/2} \mathbf 1_{k=3}\right)\]
seen as a series in $t^3$ is algebraic and has a unique dominant singularity at $t_c^3$. In view of \eqref{eq:genF}, the generating series of these modified weights is given by
\begin{equation} \label{eq:tildeF}
\tilde F(p,t,z) = \sum_{k\geq 1} \tilde q_k(p,t) \, z^{k-1} = \frac{1}{p} \, \frac{1}{1-z} T \left(1-p,t,\frac{t}{1-z}\right).
\end{equation}
Injecting this into the algebraic equation verified by $T$, we get an algebraic equation verified by $\tilde F$:
\[
\left(p \, \tilde F(p,t,z) - T(1-p,t,t) \right) \, \mathrm{Pol}_1 (p \, \tilde F(p,t,z),p,t^3,T(1-p,t,t),tT_1(p,t)) = z \cdot \mathrm{Pol}_2 (p \, \tilde F(p,t,z),p,z),
\]
where $\mathrm{Pol}_1$ and $\mathrm{Pol}_2$ are explicit polynomials. The form of this equation has the following consequences. First, using Lemma \ref{lemm:parT}, the series $\tilde F(p,t,z)$ is algebraic. Second, it is the unique solution of this equation with constant term in $z$ equal to $T(1-p,t,t)/p$, and we can compute its coefficients in $z$ inductively. These coefficients are the modified weights $\tilde q_k(p,t)$ and their expressions are then rational fractions in $p$, $t^3$, $tT_1(p,t)$ and $T(1-p,t,t)$, whose denominator (up to a factor $p$) are the $k$-th power of
\begin{align*}
\mathrm{Pol}_1 (\tilde F(p,t,0),p,t^3,T(1-p,t,t),tT_1(p,t)) &= \mathrm{Pol}_1 (T(1-p,t,t),p,t^3,T(1-p,t,t),tT_1(p,t)),\\
&= 3 t^3  (p-1) \, T^2(1-p,t,t) + p (p-1) + tT_1(t).
\end{align*}
A quick glance at equation \eqref{eq:Ty} with $y=1$ shows that the last display is the derivative of the algebraic equation verified by $T(1-p,t,t)$ with respect to $T$. Therefore this quantity can only be $0$ at singularities of $T(1-p,t,t)$, which leaves only $t^3_c$ according to Lemma \ref{lemm:parT}. As a consequence, we just proved that the series $\tilde q (p,t)$ are all algebraic series in $t^3$ and all have a unique dominant singularity at $t_c^3$.

Now that we know that $\tilde q_k(p,t)$ has a unique dominant singularity at $t_c^3$, it will follow from the general form of Puiseux expansions of algebraic series near their singularities (see \cite[Theorem VII.7, p. 49]{FS}) that $[t^{3n}] \tilde q_k(p,t)$ has the same asymptotic behavior as $[t^{3n}] \mathcal Z (p,t)$ if we can prove that there exists two positive constants $c,c'$, depending on $k$ and $p$ such that for $n$ large enough:
\[
c \cdot [t^{3n}] \mathcal Z (p,t) \leq [t^{3n}] \tilde q_k(p,t) \leq c' \cdot [t^{3n}] \mathcal Z (p,t).
\]
The upper bound is easily obtained by putting a triangulation of the $1$-gon with a white boundary vertex inside a cycle of $k$ black vertices and summing over every possible necklace between the two. The lower bound is obtained similarly starting from cycle of $k$ black vertices by putting a single white vertex inside the cycle and an edge joining this additional vertex to every boundary vertex, and putting an arbitrary triangulation with white boundary together with a matching necklace on the outside of the cycle. To sum up, we have proved that for every $k$, $\tilde q_k(p,t)$ seen as a series in $t^3$ has a unique dominant singularity at $t_c^3$ and that its asymptotic expansion at $t_c^3$ is of the form
\begin{align}
\tilde q_k(p,t) = \tilde q_k(p,t_c) - \tilde a_k (p) \left(1 - \frac{t^3}{t^3_c} \right) + \tilde b_k(p) \left(1 - \frac{t^3}{t^3_c} \right)^{3/2} + o \left(1 - \frac{t^3}{t^3_c} \right)^{3/2}.
\end{align}
This finishes the proof proof of the first statement \eqref{eq:asymptot3nqk} of the Lemma and we now have to identify the generating series of the numbers $\delta_k(p) = \tilde b_k(p) / \kappa(p)$, where $\kappa (p)$ is the coefficient of the term $\left(1 - \frac{t^3}{t^3_c} \right)^{3/2}$ in the asymptotic expansion of $\mathcal Z(p,t)$ calculated in Proposition \ref{prop:Ztcoef}.

\bigskip

Using the rational parametrization of Lemma \ref{lemm:parT}, we can find a rational parametrization for $\tilde F$. Indeed, if $V = V(1-p,U,y)$ is the power series in $\mathbb Q[p,U]  \llbracket y \rrbracket \subset \mathbb Q[p] \llbracket t^3,y \rrbracket$ defined in  Lemma \ref{lemm:parT}, we have
\[
\begin{split}
\tilde F(p,t,z) &= \frac{1}{p} \hat y \left( 1-p,U(t^3),V \Big(1-p,U(t^3),1/(1-z) \Big) \right) \\
& \qquad \times \hat T\left( 1-p,U(t^3),V \Big( 1-p,U(t^3),1/(1-z) \Big) \right),
\end{split}
\]
where $\hat y$ and $\hat T$ are rational fractions defined in  Lemma \ref{lemm:parT}. For $z$ fixed such that $|1/(1-z)| < y_+(1-p,t_c)$ (which includes a neighborhood of $0$ for $z$ since $y_+(1-p,t_c) > 1$, the series $\tilde F(p,t,z)$ seen as a series in $t^3$ has non negative coefficients and has radius of convergence $t^3_c$. In addition, this implies that $(t^3,z ) \mapsto \tilde F(p,t,z)$ is analytic in the larger domain $\mathcal D(0,t^3_c) \times \mathcal D(0,y_+(1-p,t_c)/(y_+(1-p,t_c)-1))$. We will produce an asymptotic expansion of $\tilde F(p,t,z)$ near $t^3_c$ using our rational parametrization. To do so, we start by computing the asymptotic expansion of $V \Big( p,U(t^3),1/(1-z) \Big)$ near $t^3_c$, with $z$ fixed.

First, writing
\[
\begin{split}
\frac{1}{1-z} &= \hat y \left( 1-p,U(t^3),V \Big(1-p,U(t^3),1/(1-z) \Big) \right)\\
& =  \hat y \left( 1-p,U(t^3_c),V \Big(1-p,U(t^3_c),1/(1-z) \Big) \right)
\end{split}
\]
we get an algebraic equation between $V_c^z = V \Big(1-p,U(t^3_c),1/(1-z) \Big)$, $V \Big(1-p,U(t^3),1/(1-z) \Big)$ and $U(t^3)$. Plugging the asymptotic expansion \eqref{eq:singU} of $U(t^3)$ in this equation, we obtain an asymptotic expansion for $V \Big(p,U(t^3),1/(1-z) \Big)$ of the form:
\begin{align*}
V  \Big(1-p,U(t^3),1/(1-z) \Big) & = V_c^z - a_1(V_c^z) \left(1 - \frac{t^3}{t^3_c} \right)^{1/2} + a_2(V_c^z) \left(1 - \frac{t^3}{t^3_c} \right)\\
& \qquad \qquad \qquad + a_3(V_c^z) \left(1 - \frac{t^3}{t^3_c} \right)^{3/2} + o \left(1 - \frac{t^3}{t^3_c} \right)^{3/2},
\end{align*}
where the $a_i$'s are explicit rational fractions whose expressions are given in the Maple companion file \cite{Maple}. Injecting in turn the asymptotic expansions in $t^3$ of $U$ and $V$ in $\hat T$ and $\hat y$ we find an asymptotic expansion for $\tilde F$ of the form:
\[
\tilde F (p,t,z) = \tilde F (p,t_c,z) + \tilde A(p,V_c^z) \left(1 - \frac{t^3}{t^3_c} \right) + \tilde B(p,V_c^z) \left(1 - \frac{t^3}{t^3_c} \right) + o \left(1 - \frac{t^3}{t^3_c} \right)^{3/2},
\]
where $\tilde A$ and $\tilde B$ are explicit rational functions, and are analytic on the disk $\mathcal D (0, y_+(1-p,t_c)/(y_+(1-p,t_c) - 1)$ (this is obvious from their expressions: $z \mapsto V_c^z$ is analytic in this region, and the poles of $\tilde A$ and $\tilde B$ fall outside it, see the Maple companion file \cite{Maple}). Note that the error term in the previous expansion is \emph{a priori} not uniform in $z$. To ensure that $\tilde A$ and $\tilde B$ are the respective generating series of the numbers $\tilde a_k$ and $\tilde b_k$, we see that, as power series in $(z,t^3)$ we have
\[
\tilde A(p,V_c^z) = \lim_{t \to t_c} \left( \tilde F (p,t,z) - \tilde F (p,t_c,z) \right) \cdot (1 - t^3/t_c^3)^{-1},
\]
and
\[
\tilde B(p,V_c^z) = \lim_{t \to t_c} \left( \tilde F (p,t,z) - \tilde F (p,t_c,z) - \tilde A(p,V_c^z) \left(1 - \frac{t^3}{t^3_c} \right) \right) \cdot (1 - t^3/t_c^3)^{-3/2}.
\]
Combined with the analycity properties of these series, this ensures that $\tilde A(p,V_c^z)$ and $\tilde B(p,V_c^z)$ are indeed the generating series of the numbers $\tilde a_k$ and $\tilde b_k$.

Finally, the generating series of the numbers $\delta_k$ is then given by
\[
\Delta(p,z) = \frac{z}{\kappa(p) \, (1-z)} \, \tilde B(p,V_c^z) = \frac{\hat y \left( 1-p,U(t^3_c),V_c^z \right) - 1}{\kappa(p)} \, \tilde B(p,V_c^z),
\]
giving the expression of the Lemma. See the Maple file \cite{Maple} for detailed computations.
\end{proof}

\bigskip

Applying the same proof to the function $T(p,t,t y)$ instead of $\tilde F (p,t,z)$ defined in Equation~\eqref{eq:tildeF} allows to establish the asymptotic behavior in $n$ of $[t^{3n}] t^k T_k(p,t)$ for every $k \geq 0$. We do not reproduce the proof as it will be almost exactly the same as the proof of Lemma~\ref{lem:serDelta}, with no additional difficulties but with the function
\[
\frac{1-p}{y} \, \tilde F \left(1-p, t , 1- \frac{1}{y} \right)
\]
instead of $\tilde F (p,t,z)$. The statement is as follows.

\begin{lemm} \label{lem:Tkasymp}
Fix $p \in (0,1)$ and $k \geq 1$. One has
\begin{equation} \label{eq:asymptonTk}
\frac{[t^{3n}] t^k T_{k} (p ,t)}{[t^{3n}] \mathcal Z ( p ,t)}
\underset{n \to \infty}{\rightarrow} \theta_k(p)
\end{equation}
where the generating series of the numbers $\theta_k(p)$ is given by
\begin{equation} \label{eq:Thetay}
\Theta (p, y) := \sum_{k \geq 0} \theta_k(p) \, y^k =  \frac{1-p}{y} \, \Delta \left( 1-p,1 - \frac 1 y \right),
\end{equation}
which is analytic on $\mathbb C \setminus \left[y_+(p,t_c) , + \infty \right)$.
\end{lemm}

\subsection{BDFG functions} \label{sec:gfsing}

We start with an integral formula for the functions $f^\bullet$ and $f^\diamond$.

\begin{lemm} \label{lem:BDFGintegral}
Fix $p \in (0,1)$ and $t \in (0,t_c]$. Then for $0< z_1 \leq z^+(p,t)$ and $0<z_2 \leq z^\diamond (p,t)$ we have
\begin{align*}
f^\bullet (p,t;z_1,z_2) &=  \sqrt{pt^3} \, 2 z_2 + \frac{t^3}{\pi} 
\int_{ \left(1 -\sqrt{pt^3} (z_2 - 2 \sqrt{z_1})\right)^{-1}} ^{ \left(1 -\sqrt{pt^3} (z_2 + 2 \sqrt{z_1})\right)^{-1}}  {dz} \,  \frac{1 - \sqrt{p t^3} z_2 - \frac{1}{z}}{2pt^3 \, z_1} \\
& \qquad \qquad \cdot \frac{T(1-p,t,t z)}{\sqrt{\left(1 - z \left(1- \sqrt{pt^3} (z_2 + 2 \sqrt{z_1}) \right) \right) \left(z \left(1  -\sqrt{pt^3} (z_2 - 2 \sqrt{z_1}) \right)  -1 \right)}} , \\
f^\diamond (p,t;z_1,z_2) &= \sqrt{pt^3} \, (2z_2 +z_1^2) +\frac{\sqrt{t^3/p}}{\pi} \, \int_{ \left(1 -\sqrt{pt^3} (z_2 - 2 \sqrt{z_1})\right)^{-1}} ^{ \left(1 -\sqrt{pt^3} (z_2 + 2 \sqrt{z_1})\right)^{-1}}  dz \\
& \qquad \qquad \cdot \frac{T(1-p,t,t z)}{\sqrt{\left(1 - z \left(1- \sqrt{pt^3} (z_2 + 2 \sqrt{z_1}) \right) \right) \left(z \left(1  -\sqrt{pt^3} (z_2 - 2 \sqrt{z_1}) \right)  -1 \right)}}.
\end{align*}
\end{lemm}

\begin{proof}
Fix $p \in (0,1)$ and $t \in (0,t_c]$.
Replacing the weights by their expression \eqref{eq:defqkt} in the definitions of the two functions gives
\begin{align*}
f^\bullet (p,t;z_1,z_2) &= \frac{1}{p} 2 (pt)^{3/2} z_2 + \frac{p t^3}{p} \sum_{k,k',l \geq 0} \binom{2k+k' +l +1}{k+1,k,k',l} (pt^3 \, z_1)^k \, (\sqrt{p t^3} \, z_2)^{k'} \, [y^l] T(1-p,t,ty), \\
f^\diamond (p,t;z_1,z_2) &= \frac{1}{p} (pt)^{3/2} (2z_2 +z_1^2) + \frac{\sqrt{p t^3}}{p} \sum_{k,k',l \geq 0} \binom{2k+k' +l}{k,k,k',l} (pt^3 \, z_1)^k \, (\sqrt{p t^3} \, z_2)^{k'} \, [y^l] T(1-p,t,ty).
\end{align*}
We can express these two functions as Hadamard products. Indeed, define the trivariate power series in $(z_1,z_2;z)$:
\begin{align*}
h^\bullet(z_1,z_2;z) &:= \sum_{k,k',l \geq 0} \binom{2k+k' +l +1}{k+1,k,k',l} (pt^3 \, z_1)^k \, (\sqrt{p t^3} \, z_2)^{k'} \, z^l ,\\
&= \frac{1}{2pt^3 \, z_1} \left( \frac{1 - z- \sqrt{pt^3} z_2}{\sqrt{\left(1-z-\sqrt{pt^3} (z_2 + 2 \sqrt{z_1}) \right) \left(1-z-\sqrt{pt^3} (z_2 - 2 \sqrt{z_1}) \right)}} -1 \right);
\end{align*}
and
\begin{align*}
h^\diamond(z_1,z_2;z) &:= \sum_{k,k',l \geq 0} \binom{2k+k' +l}{k,k,k',l} (pt^3 \, z_1)^k \, (\sqrt{p t^3} \, z_2)^{k'} \, z^l ,\\
&=  \frac{1 }{\sqrt{\left(1-z-\sqrt{pt^3} (z_2 + 2 \sqrt{z_1}) \right) \left(1-z-\sqrt{pt^3} (z_2 - 2 \sqrt{z_1}) \right)}} .
\end{align*}
Then we have
\begin{align*}
f^\bullet (p,t;z_1,z_2) &=  \sqrt{pt^3} \, 2 z_2 + t^3 \, T(1-p,t,t z) \odot h^\bullet(z_1,z_2;z) \vert_{z=1} , \\
f^\diamond (p,t;z_1,z_2) &= \sqrt{pt^3} \, (2z_2 +z_1^2) +\frac{\sqrt{p t^3}}{p} \, T(1-p,t,t z) \odot h^\diamond (z_1,z_2;z) \vert_{z=1}.
\end{align*}

We can calculate these two Hadamard products as contour integrals in a similar fashion than in the proof of Theorem \ref{th:offcrit} where we established \eqref{eq:HadamardContour}. For $(z_1,z_2) \in (0,z^+(p,t)] \times (0,z^\diamond(p,t)]$ the functions $h^\bullet$ and $h^\diamond$ are analytic in the domain $|z| < 1 - \sqrt{p t^3} (z_2 + 2 \sqrt{z_1})$ that contains the domain $|z| < 1 - \sqrt{p t^3} c_+(p,t) = 1 - \frac{1}{y_+(p,t)}$. This last domain is not empty since $y_+(p,t) >1$ from Lemma \ref{lemm:parT}. Therefore, we can represent the Hadamard product as a contour integral similar to \eqref{eq:HadamardContour} if $ \left(1 - \frac{1}{y_+(p,t)} \right)^{-1} \leq y_+(1-p,t)$. We can check that this is the case since we have explicit formulas for $y_+(p,t_c)$ and $y_+(p,t)$ is increasing in $t$ (see the Maple file \cite{Maple} for details). Therefore, for a contour $\gamma$ enclosing $0$ and a point in the interval $\left[ \left(1 - \frac{1}{y_+(p,t)} \right)^{-1} , y_+(1-p,t) \right]$, we have 
\begin{align*}
f^\bullet (p,t;z_1,z_2) &=  \sqrt{pt^3} \, 2 z_2 + \frac{t^3}{2i\pi} 
\oint_\gamma \frac{dz}{z} T(1-p,t,t z) \, h^\bullet (z_1,z_2,1/z) ,\\
f^\diamond (p;z_1,z_2) &= \sqrt{pt^3} \, (2z_2 +z_1^2) + \frac{\sqrt{t_c^3/p}}{2i\pi} 
\oint_\gamma \frac{dz}{z} T(1-p,t,t z) \, h^\diamond (z_1,z_2,1/z).
\end{align*}
We obtain the expressions of the Lemma after simplifications and taking contours converging to the cut $\left[ {\left(1 -\sqrt{pt^3} (z_2 - 2 \sqrt{z_1})\right)^{-1}} , { \left(1 -\sqrt{pt^3} (z_2 + 2 \sqrt{z_1})\right)^{-1}} \right]$.
\end{proof}

When $p=1/2$ and $t=t_c$, we can compute explicitly an asymptotic expansion of $f^\bullet$ and $f^\diamond$ at the point $(z^+(p,t),z^\diamond(p,t))$ that will be used in the proof of Theorem \ref{th:volume}.

\begin{lemm} \label{lem:BDFGexpansion}
Write $f^\bullet(z_1,z_2) = f^\bullet(1/2,t_c;z_1,z_2)$, $f^\diamond(z_1,z_2) = f^\diamond(1/2,t_c;z_1,z_2)$, $z^+ = z^+(1/2,t_c)$ and $z^\diamond = z^\diamond(1/2,t_c)$. Then $z_+ = \frac{27 \sqrt{3}}{32}$ and $z^\diamond = \frac{3^{1/4} \sqrt{2}}{4}$ and we have the following asymptotic expansions at $(z^+,z^\diamond)^-$:
\begin{align*}
f^\bullet (z_1,z_2) &= 1 - \frac{1}{z^+} + \frac{1}{z^+} \left( \frac{1}{z^+} - \sqrt{z_+} \partial_{z_1} f^\diamond(z^+,z^\diamond) \right) \, \left( z_1 - z^+ \right) + \partial_{z_1} f^\diamond(z^+,z^\diamond) \, \left( z_2 - z^\diamond \right) \\
& \qquad 
+ \kappa^\bullet \,
\left( (z^+ - z_1) + \sqrt{z^+} (z^\diamond - z_2) \right)^{7/6} +  o \left( \left((z^+ - z_1) + \sqrt{z^+} (z^\diamond - z_2) \right)^{7/6}
 \right),\\
f^\diamond(z_1,z_2) &= z^\diamond + \partial_{z_1} f^\diamond(z^+,z^\diamond) \, \left( z_1 - z^+ \right) + \left( 1 - \sqrt{z^+} \partial_{z_1} f^\diamond(z^+,z^\diamond) \right) \, \left( z_2 - z^\diamond \right) \\
& \qquad 
+ \kappa^\diamond \,
\left( (z^+ - z_1) + \sqrt{z^+} (z^\diamond - z_2) \right)^{7/6} +  o \left( \left((z^+ - z_1) + \sqrt{z^+} (z^\diamond - z_2) \right)^{7/6}
 \right),
 \end{align*}
where
\begin{align*}
\kappa^\diamond = \frac{4 \,2^{\frac{1}{6}} \Gamma \! \left(\frac{2}{3}\right)^{3} 3^{\frac{2}{3}} \sqrt{5}}{63 \pi^{2}} \quad \text{and} \quad \kappa^\bullet = \frac{512 \, 3^{\frac{1}{4}} \, \sqrt{2}}{81} \kappa^\diamond.
\end{align*}
\end{lemm}

\begin{proof}
The respective values of $z^+$ and $z^\diamond$ are computed from \eqref{eq:cpm} and the explicit values of the singularities $y_+(1/2,t_c) = 2$ and $y_-(1/2,t_c) = -4$ of the function $y \mapsto T(1/2,t_c, t_cy)$ (see the Maple file \cite{Maple} for details):
\begin{align*}
z^+ &= \frac{c_+(1/2,t_c) + c_-(1/2,t_c)}{2} = \frac{27 \sqrt{3}}{32} ,\\
z^\diamond &= \left(\frac{c_+(1/2,t_c) - c_-(1/2,t_c)}{4}\right)^2 = \frac{3^{3/4} \sqrt{2}}{4}.
\end{align*}

Now, the change of variable 
\[
z = \phi(z_1,z_2;\xi) := \frac{1}{1-\sqrt{pt^3} (z_2 - 2 \sqrt{z_1})} + \left(\frac{1}{1-\sqrt{pt^3} (z_2 + 2 \sqrt{z_1})} - \frac{1}{1-\sqrt{pt^3} (z_2 - 2 \sqrt{z_1})} \right) \xi
\]
in the expressions of Lemma \ref{lem:BDFGintegral} gives
\begin{align*}
f^\bullet (p,t;z_1,z_2) &=  \sqrt{pt^3} \, 2 z_2 + \frac{1}{2 p z_1} \frac{1}{\pi}
\int_{ 0} ^{ 1}
d \xi 
\, \xi^{-1/2} (1-\xi)^{-1/2} \, \frac{\phi(z_1,z_2;\xi) ( 1 - \sqrt{pt^3} z_2 ) -1}{\phi(z_1,z_2;\xi)} \\
& \qquad \qquad \qquad \qquad \qquad \qquad \qquad \cdot T \left( 1-p,t, t \, \phi(z_1,z_2;\xi)\right),\\
f^\diamond (p,t;z_1,z_2) &= \sqrt{pt^3} \, (2z_2 +z_1^2) + \sqrt{\frac{t^3}{p}} \frac{1}{\pi}
\int_{0} ^{1}
d \xi 
 \, \xi^{-1/2} (1-\xi)^{-1/2} \, T \left( 1-p,t, t \, \phi(z_1,z_2;\xi) \right).
\end{align*}
We will see that, when $p=1/2$ and $t=t_c$, the main term in the expansion of these functions stems from the singular behavior of $T$ at $y_+ =2$.

\bigskip

From Lemma \ref{lem:Vcritexp} and the discussion that follows, we know that the function $y \mapsto T(1/2,t_c,t_c \, y)$ is analytic for $y\in [0,2)$ and 
has the following expansion as $y \to 2^-$:
\[
T(1/2,t_c,t_c \, y) = \frac{\sqrt{3}}{2} - \frac{3^{5/6}}{2} \left( 1- \frac{y}{2} \right)^{2/3} + \frac{\sqrt{3}}{2} \left( 1- \frac{y}{2} \right) -  3^{1/6} \left( 1- \frac{y}{2} \right)^{4/3} + \mathcal O \left(\left( 1- \frac{y}{2} \right)^{5/3} \right).
\]
The function
\[
\varphi(y) := T(1/2,t_c,t_c \, y) - \left( \frac{\sqrt{3}}{2} - \frac{3^{5/6}}{2} \left( 1- \frac{y}{2} \right)^{2/3} + \frac{\sqrt{3}}{2} \left( 1- \frac{y}{2} \right) +  3^{1/6} \left( 1- \frac{y}{2} \right)^{4/3} \right)
\]
is twice differentiable on $[0,2)$. In addition, for $\xi \in (0,1)$ and if $(z_1,z_2) \in [0,z^+]\times [0,z^\diamond]$, the quantity $\phi(z_1,z_2;\xi)$ varies in a subset of $(0,(1-1/y_+(1/2,t_c))^{-1}) = (0,2)$. We then have
\[
\int_0^1 d \xi \, \xi^{-1/2} (1-\xi)^{-1/2} \left| \varphi' \left( \phi(z_1,z_2;\xi) \right) \right| < + \infty
\]
and
\[
\int_0^1 d \xi \, \xi^{-1/2} (1-\xi)^{-1/2} \left| \varphi'' \left( \phi(z_1,z_2;\xi) \right) \right| < + \infty.
\]
As a consequence, the function
\[
\Phi (z_1 ,z_2) = \frac{1}{\pi} \int_0^1 d \xi \, \xi^{-1/2} (1-\xi)^{-1/2}  \varphi \left( \phi(z_1,z_2;\xi) \right)
\]
is twice differentiable on $[0,z^+]\times [0,z^\diamond]$ and has the following asymptotic expansion when $(z_1,z_2) \to (z^+,z^\diamond)$:
\begin{align*}
\Phi (z_1,z_2) &= \Phi(z^+, z^\diamond) + \nabla \Phi(z^+, z^\diamond) \cdot \left( z_1 -z^+, z_2 -z^\diamond \right) + \mathcal O  \left( (z_1 -z^+)^2 +( z_2 -z^\diamond)^2 \right).
\end{align*}

The singular parts of the expansions of $f^\bullet$ and $f^\diamond$ come from the singularities of the form $(1-y/2)^\alpha$ in the development of $T(1/2,t_c,t_c \, y)$ at $y=2$ for $\alpha \in \{ 2/3 , 4/3 \} $ (it is straightforward to check that the linear term $(1-y/2)$ contributes only to non singular parts in the expansion). Indeed, for $\alpha \in \{ 2/3 , 4/3 \}$, set
\begin{align*}
I_\alpha (z_1,z_2) &= \frac{1}{\pi} \int_0^1 d \xi \, \xi^{-1/2} (1-\xi)^{-1/2} \left( 1 - \frac{\phi(z_1,z_2;\xi)}{2}\right)^\alpha, \\
& = \left(1- \frac{1}{2(1-\sqrt{t_c^3/2} (z_2 - 2 \sqrt{z_1}))} \right)^\alpha \,
\frac{1}{\pi} \int_0^1 d \xi \, \xi^{-1/2} (1-\xi)^{-1/2} \\
& \qquad \qquad \qquad \qquad \qquad \qquad \qquad \qquad \cdot \left( 1 - \xi 
\frac{ \frac{1-\sqrt{t_c^3/2} (z_2 - 2 \sqrt{z_1})}{1-\sqrt{t_c^3/2} (z_2 + 2 \sqrt{z_1})} -1 }{1-2\sqrt{t_c^3/2} (z_2 - 2 \sqrt{z_1})}
\right)^\alpha, \\
& =  \left(1- \frac{1}{2(1-\sqrt{t_c^3/2} (z_2 - 2 \sqrt{z_1}))} \right)^\alpha \, _2F_1 \left( - \alpha, \frac{1}{2};1 ; \frac{\frac{1-\sqrt{t_c^3/2} (z_2 - 2 \sqrt{z_1}) }{1-\sqrt{t_c^3/2} (z_2 + 2 \sqrt{z_1})} -1 }{1-2\sqrt{t_c^3/2} (z_2 - 2 \sqrt{z_1})} \right),
\end{align*}
where we used Euler's integral representation of the hypergeometric function $_2F_1$ in the last line. This last equality is valid when $0< z_1 \leq z^+$ and $0<z_2 \leq z^\diamond$ since in this case $\sqrt{t_c^3/2} (z_2 + 2 \sqrt{z_1}) \leq 1/2$ and the variable in the hypergeometric function is in $(0,1]$.  Furthermore, using the values of $z^+$ and $z^\diamond$, a simple computation done in the Maple file \cite{Maple} gives:
\[
\frac{\frac{1-\sqrt{t_c^3/2} (z_2 - 2 \sqrt{z_1})}{1-\sqrt{t_c^3/2} (z_2 + 2 \sqrt{z_1})} -1}{1-2\sqrt{t_c^3/2} (z_2 - 2 \sqrt{z_1})} = 1 - \frac{20 \sqrt 3}{81} \left( (z^+ - z_1) + \sqrt{z^+} (z^\diamond - z_2) \right) + \mathcal O \left( (z^+ - z_1)^2 + (z^\diamond - z_2)^2 \right),
\]
Using the standard asymptotic development of hypergeometric functions at $1$ we see that the first singular term in the development of $I_{2/3}$ is
\[
- \kappa 
\left( (z^+ - z_1) + \sqrt{z^+} (z^\diamond - z_2) \right)^{7/6} \, \text{with} \, \kappa = \frac{18 \,2^{\frac{1}{3}} \, \Gamma \! \left(\frac{2}{3}\right)^{3}}{7 \pi^{2}} \, \left(\frac 3 5 \right)^{2/3} \, \left(  \frac{20 \sqrt 3}{81} \right)^{7/6}.
\]
The first singular term of $I_{4/3}(z_1,z_2)$ is similar, but with exponent $11/6$ instead of $7/6$. This means that the first singular term in the development of $f^\diamond(1/2,z_1,z_2)$ is from $- \frac{3^{5/6}}{2} \, \sqrt{t_c^3/p} \, I_{2/3}(z_1,z_2)$ and we have
\begin{align*}
f^\diamond(z_1,z_2) &= f^\diamond(z^+,z^\diamond) + \nabla f^\diamond(z^+,z^\diamond) \cdot \left( z_1 - z^+ , z_2 - z^\diamond \right) \\
&
+ \frac{3^{5/6}}{2} \, \sqrt{2 t_c^3} \, \kappa \,
\left( (z^+ - z_1) + \sqrt{z^+} (z^\diamond - z_2) \right)^{7/6} + o \left( \left((z^+ - z_1) + \sqrt{z^+} (z^\diamond - z_2) \right)^{7/6}
 \right).
\end{align*}
The statement for $f^\diamond$ follows using the fact that $f^\diamond ( z^+, z^\diamond) = z^\diamond$ and the criticality equation \eqref{eq:critdiff}.

The expansion for $f^\bullet (1/2,z_1,z_2)$ is obtained similarly by replacing $\varphi$ by
\begin{align*}
&\frac{1 - \sqrt{p t_c^3} z_2 - \frac{1}{y}}{2pt_c^3 \, z_1}
T(1/2,t_c,t_c \, y) \\
& - \frac{1 - \sqrt{t_c^3/2} z_2 - \frac{1}{2}}{t_c^3 \, z_1} \left( \frac{\sqrt{3}}{2} - \frac{3^{5/6}}{2} \left( 1- \frac{y}{2} \right)^{2/3} + \frac{\sqrt{3}}{2} \left( 1- \frac{1}{2 t_c^3 z_1} \right) \left( 1- \frac{y}{2} \right) -  3^{1/6} \left( 1- \frac{y}{2} \right)^{4/3} \right).
\end{align*}
The first singular term in the development of $f^\bullet$ is then the one from $ - \frac{3^{5/6}}{2} \, \frac{1 - \sqrt{t_c^3/2} z^\diamond - \frac{1}{2}}{t_c^3 \, z^+} \frac{1}{z^+} I_{2/3}$ and we have
\begin{align*}
f^\bullet(z_1,z_2) &= f^\bullet(z^+,z^\diamond) + \nabla f^\bullet(z^+,z^\diamond) \cdot \left( z_1 - z^+ , z_2 - z^\diamond \right) \\
& \qquad
+
\frac{3^{5/6}}{2} \, \frac{1 - 2\sqrt{t_c^3/2} z^\diamond}{2 t_c^3 \, (z^+)^2} \,
\kappa 
\left( (z^+ - z_1) +\sqrt{z^+} (z^\diamond - z_2) \right)^{7/6} \\
& \qquad +  o \left( \left( (z^+ - z_1) + \sqrt{z^+} (z^\diamond - z_2) \right)^{7/6}
 \right).
\end{align*}
The statement for $f^\bullet$ the follows from $f^\bullet (z^+,z^\diamond) = 1 - \frac{1}{z^+}$ and from the generic properties $\partial_{z_2} f^\bullet = \partial_{z_1} f^\diamond$ and $z_1 \partial_{z_1} f^\bullet + f^\bullet = \partial_{z_2} f^\diamond$.
\end{proof}

\subsection{Perimeter asymptotics at criticality} \label{sec:expansionscrit}

We gather in this Section asymptotics of several quantities appearing in this paper when $p=1/2$ and $t=t_c$. They are all consequences of the following Lemma characterizing the singularities of the function $y \mapsto V(1/2,U(t_c^3),y)$ defined in Lemma \ref{lemm:parT}.

\begin{lemm} \label{lem:Vcritexp}
The function $y \mapsto V(1/2,U(t_c^3),y)$ is analytic on $\mathbb C \setminus \left( (-\infty , -4] \cup [2, +\infty) \right)$. in addition it has the following asymptotic expansion in a slit neighborhood of $2$:
\begin{equation} \label{eq:expVcrit}
V(1/2, U(t_c^3), y ) = 
\frac{\sqrt 3}{3}
- \frac{1}{3^{\frac{1}{3}}} \left( 1 - \frac y 2 \right)^{\frac{1}{3}}
+\frac{1}{3} \left( 1 - \frac y 2 \right) -\frac{1}{3^{\frac{4}{3}}} \left( 1 -\frac y 2 \right)^{\frac{4}{3}} + \mathcal O \left( \left( 1 - \frac y 2 \right)^{\frac{5}{3}} \right).
\end{equation}
\end{lemm}
\begin{proof}
From Lemma \ref{lemm:parT}, we already know that $y \mapsto V(1/2,U(t_c^3),y)$ is analytic on the domain $\mathbb C \setminus \left( (-\infty , y_-(1/2,t_c)] \cup [y_+(1/2,t_c), +\infty) \right)$. We also know that $y_\pm(1/2,t_c)$ are the values of $\hat y (1/2, U(t_c^3),V)$ at $V=V_\pm(1/2,t_c)$ the two stationary points of $\hat y$ enclosing $0$. We can easily compute the corresponding values in the Maple file \cite{Maple}. The expansion is then easily obtained by singular inversion.
\end{proof}

\bigskip

This Lemma combined with the rational expressions in terms of $V$ that we have allow to immediately compute asymptotic expansions. Indeed, the expressions $\hat T (1/2,U(t_c^3),V)$ and $\hat \Delta (1/2,V)$ defined respectively in Lemma \ref{lemm:parT} and Lemma \ref{lem:serDelta} are singular only when $V$ is singular, and we can get an asymptotic expansion at their unique dominant singularity by plugging the expansion of $V$ in their expression. As usual, calculations are available in the Maple companion file \cite{Maple}.

\bigskip

We get the following expansion for $T$:
\begin{equation} \label{eq:Ttcritexp}
T(1/2,t_c,t_c \, y) = \frac{\sqrt{3}}{2} - \frac{3^{5/6}}{2} \left( 1- \frac{y}{2} \right)^{2/3} + \frac{\sqrt{3}}{2} \left( 1- \frac{y}{2} \right) -  3^{1/6} \left( 1- \frac{y}{2} \right)^{4/3} + \mathcal O \left(\left( 1- \frac{y}{2} \right)^{5/3} \right),
\end{equation}
and as a consequence
\begin{equation} \label{eq:Ttcritexpcoefk}
t_c^k T_k(1/2,t_c) \underset{k \to \infty}{\sim} \frac{-3^{5/6}}{2 \Gamma(-2/3)} \, 2^k \, k^{-5/3}.
\end{equation}
Similarly, using the expression \eqref{eq:tildeF} gives the asymptotic expansion of the weights $\sqrt{t_c^3/2}^{-k} q_k(1/2,t_c)$:
\begin{equation} \label{eq:tildeFcritexp}
\tilde F(1/2,t_c,z) = 2\, \sqrt{3} - 2 \, 3^{5/6} \left( 1- 2 z \right)^{2/3} + + \mathcal O \left(\left( 1- 2z \right)^{5/3} \right),
\end{equation}
and as a consequence
\begin{equation} \label{eq:qkcritexp}
\sqrt{t_c^3/2}^{-k} q_k(1/2,t_c) \underset{k \to \infty}{\sim} \frac{-2 \, 3^{5/6}}{ \Gamma(-2/3)} \, 2^{-k} \, k^{-5/3}.
\end{equation}
Note that these two singular expansions and the corresponding asymptotics were established in~\cite{BeCuMie} using different, and more involved techniques. The rational parametrization that we have simplifies this analysis a lot.

\bigskip

Using $\hat \Delta$, we get the following expansion:
\begin{equation} \label{eq:deltacritexp}
\Delta \left( 1/2 ,  z \right) \underset{z \to 1/2}{\sim} \left( \frac{32 \,3^{\frac{5}{6}}}{351}+\frac{8 \,3^{\frac{1}{3}}}{117} \right) \, \left( 1 - 2 z \right)^{-4/3},
\end{equation}
giving
\begin{equation} \label{eq:deltakexp}
\delta_k(1/2)
\underset{k \to \infty}{\sim}
\frac{1}{\Gamma(4/3)} \left( \frac{32 \,3^{\frac{5}{6}}}{351}+\frac{8 \,3^{\frac{1}{3}}}{117} \right) 2^{-k} \, k^{1/3}.
\end{equation}
Finally, using the expression \eqref{eq:Thetay}, we get
\begin{equation} \label{eq:thetacritexp}
\Theta \left( 1/2 ,  y \right) \underset{y \to 2}{\sim} \left( \frac{8 \,3^{\frac{5}{6}}}{351}+\frac{2 \,3^{\frac{1}{3}}}{117} \right) \, \left( 1 - y/2 \right)^{-4/3},
\end{equation}
giving
\begin{equation} \label{eq:thetakexp}
\theta_k(1/2)
\underset{k \to \infty}{\sim}
\frac{1}{\Gamma(4/3)} \left( \frac{8 \,3^{\frac{5}{6}}}{351}+\frac{2 \,3^{\frac{1}{3}}}{117} \right) 2^{k} \, k^{1/3}.
\end{equation}

\bibliographystyle{plain}
\bibliography{Perco}

\end{document}